\RequirePackage{booktabs}
\documentclass[pdflatex,sn-mathphys,Numbered]{sn-jnl} 

\usepackage{graphicx}%
\usepackage{multirow}%
\usepackage{amsmath,amssymb,amsfonts}%
\usepackage{amsthm}%
\usepackage{mathrsfs}%
\usepackage{mathdots}
\usepackage[title]{appendix}%
\usepackage{xcolor}%
\usepackage{textcomp}%
\usepackage{manyfoot}%
\usepackage{array,booktabs}%
\usepackage{algorithm}%
\usepackage{algorithmicx}%
\usepackage{algpseudocode}%
\usepackage{listings}%
\usepackage{adjustbox}
\usepackage{mathtools}
\usepackage{float}
\usepackage{subfigure}
\usepackage{epstopdf}
\usepackage{diagbox}
\usepackage{lineno}
\usepackage{hhline}
\usepackage{bm}
\usepackage{longtable}
\usepackage{lipsum}
\usepackage{footnote}
\usepackage[misc]{ifsym}

\newtheorem{theorem}{Theorem}[section]
\newtheorem{remark}{Remark}[section]%

\newtheorem{definition}{Definition}[section]%

\newcommand{\T}{\text{T}}
\newcommand{\BB}{Barzilai-Borwein }
\newcommand{\BBB}{Barzilai and Borwein }

\newcommand\blfootnote[1]{%
	\begingroup 
	\renewcommand\thefootnote{}\footnote{#1} 
	\addtocounter{footnote}{-1}%
	\endgroup 
}

\begin{document}
	
	\title[Article Title]{Regularized Barzilai-Borwein method}
	
	
	\author{Congpei An\textsuperscript{\rm 1,2}, 
		Xin Xu\textsuperscript{\rm 1}
	}
	
	
	\abstract{In unconstrained optimization problems, gradient descent method is the most basic algorithm, and its performance is directly related to the step size. In this paper, we develop a family of gradient step sizes based on \BB method, named regularized \BB (RBB) step sizes. We indicate that the reciprocal of the RBB step size is the close solution to an $\ell_{2}^{2}$-regularized least squares problem. We propose an adaptive regularization parameter scheme based on the principle of the alternate \BB (ABB) method and the local mean curvature of the objective function. We introduce a new alternate step size criterion into the ABB method, forming a three-term alternate step size, thereby establishing an enhanced RBB method for solving quadratic and general unconstrained optimization problems efficiently. We apply the proposed algorithms to solve typical quadratic and non-quadratic optimization problems, and further employ them to address spherical $t$-design, which is a nonlinear nonconvex optimization problem on an Oblique manifold.}
	
	
	
	
	\maketitle
	
\blfootnote{\Letter Xin Xu \par \; xustonexin@gmail.com} 
\blfootnote{Congpei An \par \;
	andbachcp@gmail.com; ancp@gzu.edu.cn}\par
\footnotetext[1]{School of Mathematics, Southwestern University of Finance and Economics, Chengdu 611130, Sichuan, China}
\footnotetext[2]{Permanent address: School of Mathematics and Statistics, Guizhou University, Guiyang 550025, Guizhou, China}	

\section{Introduction}\label{sec1}

In this paper, we consider solving the large scale unconstrained optimization problem
\begin{equation}\label{generalqua}
\min_{\mathbf{x}\in\mathbb{R}^{n}} f(\mathbf{x}),
\end{equation}
where $f:\mathbb{R}^n\longrightarrow \mathbb{R}$ is a sufficiently smooth function. A minimizer is denoted by $\mathbf{x}_{*}$. Gradient descent methods have been widely used for solving \eqref{generalqua} by the following iterative form 
\begin{equation}\label{gradient method}
\mathbf{x}_{k+1}=\mathbf{x}_{k}+\frac{1}{\alpha_{k}}(-\mathbf{g}_{k}),
\end{equation}
where $\mathbf{x}_{k}$ is the $k$th approximation to $\mathbf{x}_{*}$, $\mathbf{g}_{k}:=\nabla f(\mathbf{x}_{k})$ is the gradient of $f$ at $\mathbf{x}_{k}$, $\frac{1}{\alpha_{k}}$ is the step size. Different gradient descent methods would have different rules for determining the  $\alpha_{k}$. The classical steepest descent (SD) method proposed by Cauchy \cite{Cauchy2009Methodegeneralepour} determines its step size by the so-called exact line search
\begin{equation*}\label{SDG}
\alpha_{k}^{SD}=\mathop{\text{argmin}}_{\alpha>0} f\big(\mathbf{x}_{k}+\frac{1}{\alpha}(-\mathbf{g}_{k})\big).
\end{equation*}
The convergence rate of the SD method depends strongly on the morphology of the objective function $f$. If the ratio of the maximum to minimum eigenvalue of the Hessian matrix of $f$ at any local minimizer $\mathbf{x}_{*}$ is large, the method generates a zigzagging path in a neighborhood of $\mathbf{x}_{*}$, leading to poor performance  \cite{Akaike1959successivetransformationprobability,Sun2006OptimizationTheoryMethods,JorgeNocedal2006NumericalOptimization}.  

In 1988, \BBB(BB)  \cite{Barzilai1988TwoPointStep} proposed two novel $\alpha_{k}$ as follows
\begin{equation}\label{BB steps}
\alpha_{k}^{BB1}:=\frac{\mathbf{s}_{k-1}^{\T}\mathbf{y}_{k-1}}{\mathbf{s}_{k-1}^{\T}\mathbf{s}_{k-1}}\quad\text{and}\quad \alpha_{k}^{BB2}:=\frac{\mathbf{y}_{k-1}^{\T}\mathbf{y}_{k-1}}{\mathbf{s}_{k-1}^{\T}\mathbf{y}_{k-1}},
\end{equation} 
which are the solutions to the following least squares models:
\begin{equation}\label{LS}
\min_{\alpha\in\mathbb{R}} \|\alpha \mathbf{s}_{k-1} - \mathbf{y}_{k-1}\|_{2}^{2}\quad\text{and}\quad \min_{\alpha\in \mathbb{R}} \|\mathbf{s}_{k-1} -  \frac{1}{\alpha}\mathbf{y}_{k-1}\|_{2}^{2},
\end{equation} 
respectively, where $\mathbf{s}_{k-1}=\mathbf{x}_{k}-\mathbf{x}_{k-1}$ and $\mathbf{y}_{k-1}=\mathbf{g}_{k}-\mathbf{g}_{k-1}$. Assuming $\mathbf{s}_{k-1}^{\T}\mathbf{y}_{k-1}>0$, by the Cauchy-Schwarz inequality \cite{Golub2013MatrixComputations}, one has $\alpha_{k}^{BB1}\le\alpha_{k}^{BB2}$. Hence, $\frac{1}{\alpha_{k}^{BB1}}$ is often called the \textit{long} BB step size while $\frac{1}{\alpha_{k}^{BB2}}$ is called the \textit{short} BB step size \cite{Huang2022accelerationBarzilaiBorweinmethod}. In the least squares sense, the BB method approximates the Hessian of $f(\mathbf{x}_{k})$ using $\alpha_{k}^{BB1}\mathbf{I}$ or $\alpha_{k}^{BB2}\mathbf{I}$. Therefore, the BB method incorporates quasi-Newton approach to gradient method by the choice of step size. Moreover, extensive numerical results in  \cite{Friedlander1998GradientMethodRetards,Fletcher2005BarzilaiBorweinMethod,Yuan2018StepSizesGradient} show that the performance of the BB method is significantly superior to that of the SD method. The success of the BB method demonstrates that the poor performance of the SD method cannot be attributed to its search direction; rather, the key factor lies in the step size.  

In the seminal paper \cite{Barzilai1988TwoPointStep}, the authors applied BB method to solve the strictly convex quadratic optimization problem. Especially when $n=2$, the authors proved that the BB method converges \text{R-superlinearly} to the global minimizer. In any dimension, it is still globally convergent \cite{Raydan1993BarzilaiBorweinchoice} but the convergence is R-linear  \cite{Dai2002Rlinearconvergence}. 

In practice, the BB method has garnered significant attention in the optimization community due to its low computational complexity and superior performance. Subsequently, a series of in-depth studies and extensions related to the BB method have emerged. For example, \cite{Molina1996PreconditionedBarzilaiBorwein} proposed the preconditioned \BB method and applied to the numerical solution of partial differential equations.  \cite{Fletcher2005BarzilaiBorweinMethod} provided a detailed analysis of the principle behind the BB method and demonstrated that its performance on certain perturbed problems is comparable to that of the conjugate gradient method. \cite{Dai2005AsymptoticBehaviourSome} investigated the asymptotic behavior of the BB method in high-dimensional optimization problems. Inspired by the BB method, \cite{Friedlander1998GradientMethodRetards} proposed a gradient descent method with delayed steps. \cite{Zhou2006GradientMethodsAdaptive} introduced a gradient descent method with alternating BB (ABB) step sizes, which significantly improves the performance of the original BB method by setting a threshold to dynamically switch between the long and short BB steps. \cite{Dai2006cyclicBarzilaiBorwein} developed a cyclic BB method that reuses a single step size over consecutive iterations, fully leveraging the spectral properties of the BB approach. \cite{Huang2021EquippingBarzilaiBorwein,Huang2022accelerationBarzilaiBorweinmethod} explored accelerating the BB method by enforcing step sizes with finite termination properties. Recently, \cite{Li2024familyBarzilaiBorwein} derived a class of BB step sizes from the perspective of scaled total least squares.

Raydan  \cite{Raydan1997BarzilaiBorweinGradienta} adapted the BB method to non-quadratic unconstrained optimization by incorporating the non-monotone line search of Grippo et al. \cite{Grippo1986NonmonotoneLineSearch}. Since this work, the BB method has been successfully extended to many fields such as convex constrained optimization \cite{Dai2005ProjectedBarzilaiBorweina,Dai2005Newalgorithmssingly,Serafino2018TwoPhaseGradient,Crisci2020SpectralPropertiesBarzilaia}, nonlinear least squares \cite{Mohammad2018StructuredTwoPoint}, image processing \cite{Bonettini2009scaledgradientprojection,Jalilian2023}, spherical $t$-design \cite{An2020Numericalconstructionspherical},  etc.

For strictly convex quadratic optimization problems, the BB method decreases the objective function value in a nonmonotonic manner, but this nonmonotonicity does not undermine its overall convergence. Nevertheless, it is noteworthy that such nonmonotonicity can induce drastic oscillations in the gradient norm, preventing the algorithm from converging rapidly, particularly in the BB1 method with long step sizes. As pointed out in \cite{Fletcher2012limitedmemorysteepest}, maintaining monotonicity is important for gradient descent methods, especially when minimizing general objective functions. 

In this paper, we do not aim to impose strict monotonicity on the BB method but seek to enhance its monotonicity by modifying the BB1 step size. The original BB step size is derived by solving the corresponding least squares model, and regularization \cite{Golub1999TikhonovRegularizationTotal,Lu2013RegularizationTheoryIll} methods are often used to improve the solutions of least squares problems. Building on this insight, we propose incorporating a regularization term into the least squares model associated with the BB1 step size and adaptively adjusting the regularization parameter to improve the performance of the BB1 method.

The main contributions of this paper are as follows: We formally integrate regularization into the BB1 method, resulting in a class of regularized BB (RBB) step sizes, which offers a new perspective for refining BB-like methods. Additionally, we analyze the mathematical principles behind the ABB method \cite{Zhou2006GradientMethodsAdaptive} and propose an effective scheme for selecting the regularization parameter based on this analysis and the local mean curvature of the objective function. Furthermore, building on our analysis of the ABB method, we introduce a new alternate step size criterion and incorporate it into the ABBmin method, thereby developing an enhanced RBB method. Finally, we conduct numerical experiments to apply the proposed RBB methods to solve spherical $t$-design problems.

The remaining part of this paper is organized as follows. In Section \ref{sectionRegBB}, we propose the  regularized BB method. In Section \ref{sec:parameter}, we first analyze the principle of the ABB method and then introduce a family of effective regularization parameters based on this principle and the local mean curvature of objective function. In Section \ref{sec:ERBB}, we propose an enhanced regularized BB method. In Section \ref{sec:numerical}, we conduct numerical experiments to verify the effectiveness of the proposed methods. We conclude in the last section. Unless otherwise stated, throughout this paper $\|\cdot\|_{2}$ refers the Euclidean $2$-norm of vectors, $\mathbf{1}\in\mathbb{R}^{n}$ denotes the all-ones vector, $\mathbf{0}\in\mathbb{R}^{n}$ denotes the all-zeros vector, and $\kappa(\mathbf{A})=\frac{\lambda_{n}}{\lambda_{1}}$ denotes the condition number of a matrix $\boldsymbol{A}$, where $\lambda_{1}$ and $\lambda_{n}$ are the smallest and largest eigenvalues of $\boldsymbol{A}$.

\section{Regularized BB method for convex quadratics}\label{sectionRegBB}

In this section, we consider the following quadratic function
\begin{equation}\label{pro:quaa}
f(\mathbf{x})=\frac{1}{2}(\mathbf{x}-\mathbf{x}_{*})^{\T}\mathbf{A}(\mathbf{x}-\mathbf{x}_{*}),
\end{equation}
where $\mathbf{A}$ is a symmetric positive definite (SPD) matrix. Since gradient descent method is invariant under rotational and translational transformations, without loss of generality, we assume that $\mathbf{A}$ is a diagonal matrix:
\begin{equation}\label{A}
	\mathbf{A} = \text{diag}\{\lambda_{1},\ldots,\lambda_{n}\},
\end{equation}
where $0<\lambda_{1}\le\ldots,\lambda_{n}$. 

\subsection{Motivation}\label{motivation}
We investigate the convergence behavior of the BB1 and BB2 methods. Specifically, we apply these two methods to the non-stochastic problem in \cite{DeAsmundis2014efficientgradientmethod} with a diagonal $\mathbf{A}$ given by
\begin{equation}\label{diagconditionn}
\lambda_{i}=10^{\frac{ncond}{n-1}(n-i)},\quad i=1,\ldots,n,
\end{equation}
where $ncond=\text{log}_{10}(\kappa(\mathbf{A}))$. We set $\mathbf{x}_{*}=\mathbf{1}$ and the initial guess  $\mathbf{x}_{1}=\mathbf{0}$. We set $n=10$, $\kappa(\mathbf{A})=10^{5}$ and use the stopping condition $
\Vert \mathbf{g}_{k}\Vert_{2}\le 10^{-6}\Vert \mathbf{g}_{1}\Vert_{2}$. Initial \text{step size} is $\frac{\mathbf{g}_{1}^{\T}\mathbf{g}_{1}}{\mathbf{g}_{1}^{\T}\mathbf{A}\mathbf{g}_{1}}$.  

From Figure \ref{fig:Motivation}, it can be observed that in this typical problem, the history of $\alpha_{k}^{BB1}$ are concentrated around $\lambda_{10}$, with the corresponding $\|\mathbf{g}_{k}\|_{2}$ exhibiting violent oscillation. Meanwhile, the history of $\alpha_{k}^{BB2}$ present an asymptotically decreasing distribution pattern, and the corresponding values of  $\|\mathbf{g}_{k}\|_{2}$ converge in a relatively stable manner. Since
\begin{equation*}
	\mathbf{g}_{k+1}^{i}=(1-\frac{1}{\alpha_{k}}\lambda_{i})\mathbf{g}_{k}^{i}, \quad i\in\{1,\ldots,10\},
\end{equation*}
where $\lambda_{1}\le\lambda_{i}\le\lambda_{10}$, and $\alpha_{k}^{BB1}\le\alpha_{k}^{BB2}$, the results of Figure \ref{fig:Motivation} demonstrate that $\alpha_{k}^{BB1}$ is difficult to approach $\lambda_{10}$, while $\alpha_{k}^{BB2}$ is relatively easy to approximate $\lambda_{10}$. Therefore, in this problem, the BB1 method requires numerous iterations to eliminate the gradient component corresponding to $\lambda_{10}$, leading to a low convergence rate. In contract, the BB2 method demonstrates higher stability due to its ability to approximate $\lambda_{10}$, making it particularly suitable for some  ill-conditioned problems.

Nevertheless, if we blindly adopt short step sizes similar to those in the BB2 method, the stability of algorithm will be improved, but the convergence rate will also be affected because they cannot rapidly approach $\lambda_{1}$. The following system of inequalities (which derived from Cauchy-Schwarz inequality) explains these phenomena in the BB methods,
\begin{equation}\label{equ:inequ}
\lambda_{1}\le\frac{\mathbf{s}_{k-1}^{\mathrm{T}}\mathbf{y}_{k-1}}{\mathbf{s}_{k-1}^{\mathrm{T}}\mathbf{s}_{k-1}}\le \frac{\mathbf{s}_{k-1}^{\mathrm{T}}\mathbf{A}\mathbf{y}_{k-1}}{\mathbf{s}_{k-1}^{\mathrm{T}}\mathbf{A}\mathbf{s}_{k-1}}\le\ldots\le
\frac{\mathbf{s}_{k-1}^{\mathrm{T}}\mathbf{A}^{m}\mathbf{y}_{k-1}}{\mathbf{s}_{k-1}^{\mathrm{T}}\mathbf{A}^{m}\mathbf{s}_{k-1}}\le\lambda_{n},
\end{equation}
where $m\ge 1$ is an integer. The Rayleigh quotients closer to the left tend to approximate $\lambda_{1}$, while those closer to the right tend to approximate $\lambda_{n}$. The latter contributes to enhancing stability, whereas the former facilitates improving the convergence rate. It is evident that the BB method is essentially an eigenvalue approximation method.
 
\begin{figure}[!h]
	\centering
	\subfigure{
		\includegraphics[width=0.46\textwidth]{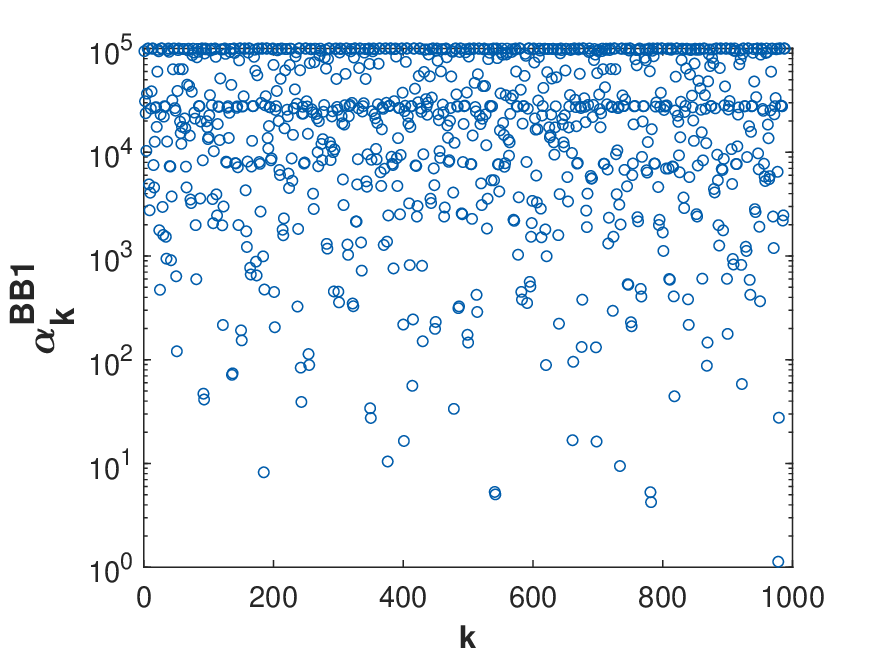}}\hspace{-2pt}
	\subfigure{
		\includegraphics[width=0.46\textwidth]{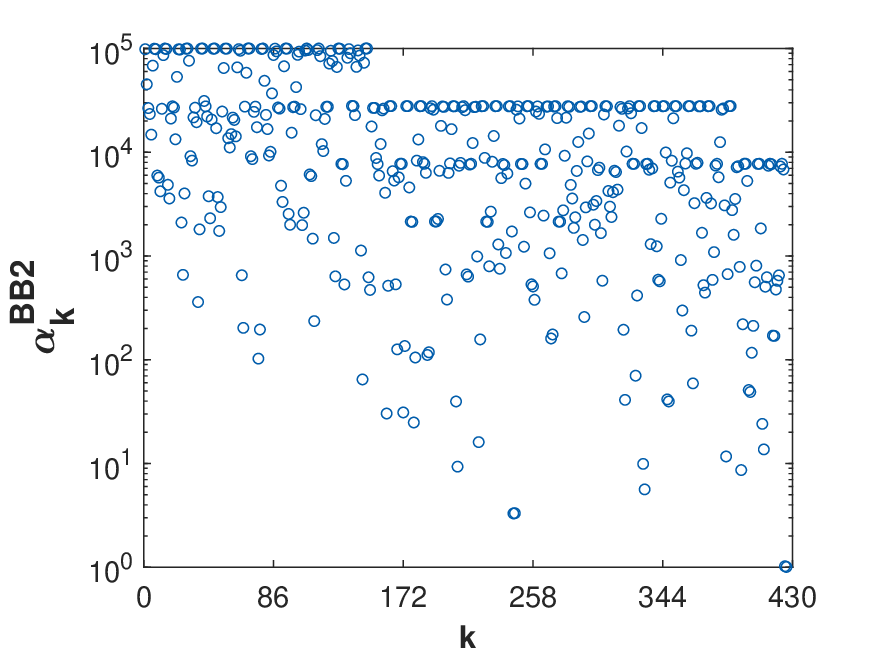}}\\
	\subfigure{
		\includegraphics[width=0.46\textwidth]{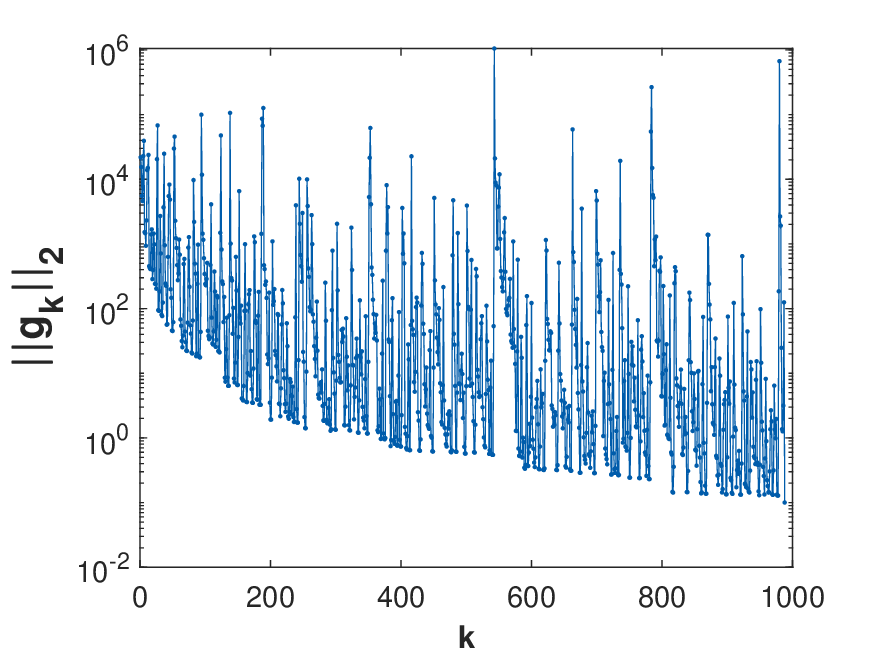}}\hspace{-2pt}
	\subfigure{
		\includegraphics[width=0.46\textwidth]{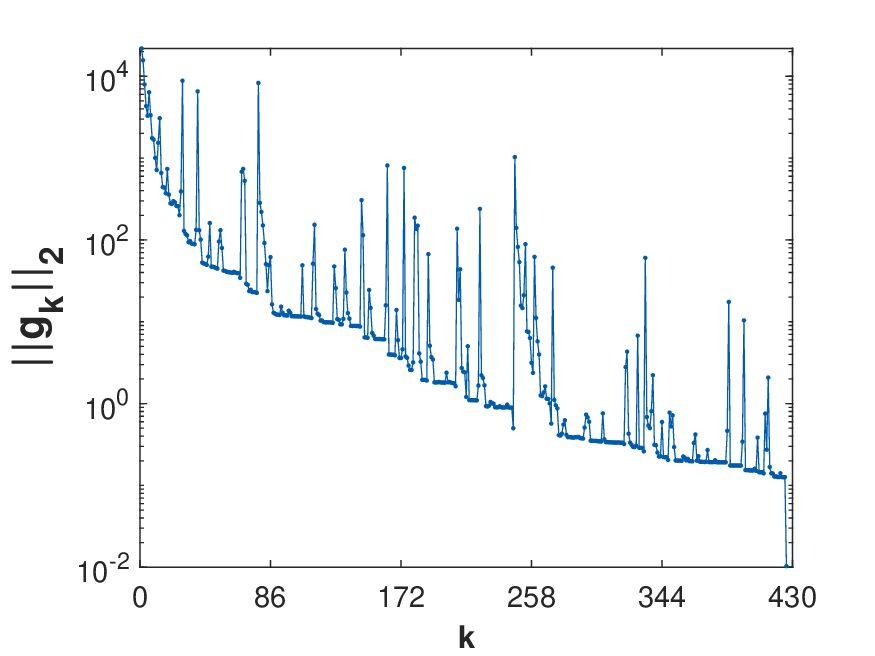}}\\
	\caption{\textit{Problem \eqref{diagconditionn} with $n=10$: historical values of $\alpha_{k}$ (top) and corresponding $\|\mathbf{g}_{k}\|_{2}$ (bottom) generated by the BB1 and BB2 methods}}	
	\label{fig:Motivation}
\end{figure}

This example shows that for strictly convex quadratic optimization problems, in a well-designed BB-like method, the sequence $\alpha_{k}^{BB}$ should exhibit an overall asymptotic convergence trend toward $\lambda_{1}$ with both rapid and stable characteristics. Based on the least squares model corresponding to the BB1 step size, we apply regularization techniques to achieve this goal. Regularization serves to balance long and short step sizes via regularization parameter. 

\subsection{Regularized \BB step sizes}
In this subsection, we derive a new class of step sizes by incorporating regularization term into the least squares model of the BB1 method, and subsequently propose an effective adaptive scheme for regularization parameter selection. 

Assuming $\Phi_{k}(\mathbf{A})$ is a symmetric positive definite matrix, we consider the following regularized least squares problem
\begin{equation}\label{equ: least square with regular}
\min_{\alpha>0} \Big\{\Vert \alpha \mathbf{s}_{k-1} - \mathbf{y}_{k-1}\Vert_{2}^{2} + \tau_{k}\Vert  \alpha \Phi_{k}(\mathbf{A}) \mathbf{s}_{k-1} - \Phi_{k}(\mathbf{A}) \mathbf{y}_{k-1}\Vert_{2}^{2} \Big\},
\end{equation}
where $\tau_{k}\ge0$ is the regularization parameter. We prove in Theorem \ref{theorem: solution} that 
\begin{equation}\label{equ: regBBstepsize}
\alpha_{k}^{RBB}=\frac{\mathbf{s}_{k-1}^{\mathrm{T}}\mathbf{y}_{k-1}+\tau_{k} \mathbf{s}_{k-1}^{\mathrm{T}}\Phi_{k}(\mathbf{A})^{\T}\Phi_{k}(\mathbf{A}) \mathbf{y}_{k-1}}{\mathbf{s}_{k-1}^{\mathrm{T}}\mathbf{s}_{k-1}+\tau_{k} \mathbf{s}_{k-1}^{\mathrm{T}}\Phi_{k}(\mathbf{A})^{\T}\Phi_{k}(\mathbf{A}) \mathbf{s}_{k-1}}
\end{equation}
is the solution to problem \eqref{equ: least square with regular}.

\begin{theorem}\label{theorem: solution}
	Let $\tau_{k}\ge0$. Then the scalar $\alpha_{k}^{RBB}$ defined by \eqref{equ: regBBstepsize} is the unique solution to problem \eqref{equ: least square with regular}.
\end{theorem}
{\it Proof} Because problem \eqref{equ: least square with regular} is a strictly convex unconstrained optimization problem, the stationary point of the objective is none other than the unique solution to \eqref{equ: least square with regular}. Taking the first derivative of the objective in problem \eqref{equ: least square with regular}  with respect to $\alpha$ leads to the first-order optimality condition
	\begin{equation}\label{firstCondition}
	\alpha(\mathbf{s}_{k-1}^{\mathrm{T}}\mathbf{s}_{k-1}+\tau_{k} \mathbf{s}_{k-1}^{\mathrm{T}}\Phi_{k}(\mathbf{A})^{\T}\Phi_{k}(\mathbf{A}) \mathbf{s}_{k-1})-(\mathbf{s}_{k-1}^{\T}\mathbf{y}_{k-1}+\tau_{k} \mathbf{s}_{k-1}^{\T}\Phi_{k}(\mathbf{A})^{\T}\Phi_{k}(\mathbf{A}) \mathbf{y}_{k-1})=0.
	\end{equation}
	With the first-order optimality condition \eqref{firstCondition}, we obtain the solution \eqref{equ: regBBstepsize}. Since $\Phi_{k}(\mathbf{A})^{\T}\Phi_{k}(\mathbf{A})$ is a symmetric positive definite matrix, the denominator of $\alpha_{k}^{RBB}$ in \eqref{equ: regBBstepsize} is not equal to zero when $\mathbf{s}_{k-1}\neq\mathbf{0}$. This completes the proof.
\qed

If $\tau_{k}=0$, \eqref{equ: regBBstepsize} degenerates to $\alpha_{k}^{BB1}$. Currently, there are two selectable options: $\tau_{k}$ and $\Phi_{k}(\mathbf{A})$, where the selection of $\Phi_{k}(\mathbf{A})$ is determined by the inequalities \eqref{equ:inequ}. We aim for the regularization term to generate a larger scalar than $\alpha_{k}^{BB1}$, then utilizing the regularization parameter to balance between long and short step sizes.  

We now consider $\Phi_{k}(\mathbf{A})=\sqrt{\mathbf{A}}$ and $\Phi_{k}(\mathbf{A})=\mathbf{A}$ in \eqref{equ: regBBstepsize}. For clarity, we designate the scalars associated with them as $\alpha_{k}^{RBB}$ and $\alpha_{k}^{RBBA}$, respectively, which implies 
\begin{equation}\label{equ: quaregBBstepsize}
\begin{split}
\alpha_{k}^{RBB}=&\frac{\mathbf{s}_{k-1}^{\T}\mathbf{y}_{k-1}+\tau_{k}\mathbf{s}_{k-1}^{\T} \mathbf{A}\mathbf{y}_{k-1}}{\mathbf{s}_{k-1}^{\T}\mathbf{s}_{k-1}+\tau_{k} \mathbf{s}_{k-1}^{\T}\mathbf{A}\mathbf{s}_{k-1}}=\frac{\mathbf{s}_{k-1}^{\T}\mathbf{y}_{k-1}+\tau_{k}\mathbf{y}_{k-1}^{\T} \mathbf{y}_{k-1}}{\mathbf{s}_{k-1}^{\T}\mathbf{s}_{k-1}+\tau_{k} \mathbf{s}_{k-1}^{\T}\mathbf{y}_{k-1}},\\ \alpha_{k}^{RBBA}=&\frac{\mathbf{s}_{k-1}^{\T}\mathbf{y}_{k-1}+\tau_{k}\mathbf{s}_{k-1}^{\T} \mathbf{A}^2\mathbf{y}_{k-1}}{\mathbf{s}_{k-1}^{\T}\mathbf{s}_{k-1}+\tau_{k} \mathbf{s}_{k-1}^{\T}\mathbf{A}^2\mathbf{s}_{k-1}}=\frac{\mathbf{s}_{k-1}^{\T}\mathbf{y}_{k-1}+\tau_{k}\mathbf{y}_{k-1}^{\T}\mathbf{A} \mathbf{y}_{k-1}}{\mathbf{s}_{k-1}^{\T}\mathbf{s}_{k-1}+\tau_{k} \mathbf{y}_{k-1}^{\T}\mathbf{y}_{k-1}}.
\end{split}
\end{equation}
For these two selections of $\Phi_{k}(\mathbf{A})$, the solutions of the regularization term $\mathop{\text{argmin}}_{\alpha}\Vert  \alpha \Phi_{k}(\mathbf{A}) \mathbf{s}_{k-1}-\Phi_{k}(\mathbf{A}) \mathbf{y}_{k-1}\Vert_{2}^{2}$ in the problem \eqref{equ: least square with regular} are
\begin{equation*}
\beta_{k}=\frac{\mathbf{s}_{k-1}^{\mathrm{T}}\mathbf{A}\mathbf{y}_{k-1}}{\mathbf{s}_{k-1}^{\mathrm{T}}\mathbf{A}\mathbf{s}_{k-1}}=\alpha_{k}^{BB2}\quad\text{and}\quad\beta_{k}^{A}=\frac{\mathbf{s}_{k-1}^{\mathrm{T}}\mathbf{A}^2\mathbf{y}_{k-1}}{\mathbf{s}_{k-1}^{\mathrm{T}}\mathbf{A}^2\mathbf{s}_{k-1}},
\end{equation*}
respectively.
 
\begin{theorem}\label{montonic}
	Assume that $\mathbf{s}_{k-1}^{\mathrm{T}}\mathbf{y}_{k-1}>0$ and $\tau_{k}\ge 0$. Then the $\alpha_{k}^{RBB}$ in \eqref{equ: quaregBBstepsize} belongs to  $[\alpha_{k}^{BB1},\, \beta_{k}]$ and is monotonically increasing with respect to parameter $\tau_{k}$.
\end{theorem}
{\it Proof}	Let $h_{\tau_{k}}(\alpha^{RBB})$ be the derivative of the $\alpha_{k}^{RBB}$  with respect to $\tau_{k}$. We have  
	\begin{equation*}
	h_{\tau_{k}}(\alpha^{RBB})=\frac{\mathbf{s}_{k-1}^{\mathrm{T}}\mathbf{A}\mathbf{y}_{k-1}\mathbf{s}_{k-1}^{\mathrm{T}}\mathbf{s}_{k-1}-\mathbf{s}_{k-1}^{\mathrm{T}}\mathbf{y}_{k-1}\mathbf{s}_{k-1}^{\mathrm{T}}\mathbf{A}\mathbf{s}_{k-1}}{\big(\mathbf{s}_{k-1}^{\mathrm{T}}(\mathbf{I}+\tau_{k} \mathbf{A})\mathbf{s}_{k-1}\big)^{2}}.
	\end{equation*} 
	According to the inequalities \eqref{equ:inequ}, we have $$\mathbf{s}_{k-1}^{\mathrm{T}}\mathbf{A}\mathbf{y}_{k-1}\mathbf{s}_{k-1}^{\mathrm{T}}\mathbf{s}_{k-1}-\mathbf{s}_{k-1}^{\mathrm{T}}\mathbf{y}_{k-1}\mathbf{s}_{k-1}^{\mathrm{T}}\mathbf{A}\mathbf{s}_{k-1}\ge0.$$ And since $\mathbf{s}_{k-1}^{\T}\mathbf{y}_{k-1}>0$, we have  $\mathbf{s}_{k-1}^{\T}(\mathbf{I}+\tau_{k}\mathbf{A})\mathbf{s}_{k-1}\neq0$. Thus $h_{\tau_{k}}(\alpha^{RBB})\ge0$. This shows that $\alpha_{k}^{RBB}$ increases monotonically  with respect to $\tau_{k}$. The proof is completed. 
\qed

\begin{remark}
For $\Phi_{k}(\mathbf{A})=\mathbf{A}$ (i.e., $\alpha_{k}^{RBBA}$ in \eqref{equ: quaregBBstepsize}), by following a similar process in Theorem \ref{montonic}, we can obtain an analogous conclusion.   	
\end{remark}

\section{Selecting regularization parameter}\label{sec:parameter}
In this section, we address another critical issue: the selection of regularization parameter. We still consider the strictly convex quadratic problems \eqref{pro:quaa}. From the results in Theorem \ref{montonic}, we know that if the regularization parameter $\tau_{k}\rightarrow 0$, then $\alpha_{k}^{RBB}\rightarrow\alpha_{k}^{BB1}$; conversely, if the regularization parameter $\tau_{k}\rightarrow\infty$, then $\alpha_{k}^{RBB}\rightarrow\alpha_{k}^{BB2}$. For brevity, we will focus on RBB hereafter, as RBBA yields analogous conclusions. 

\subsection{Motivation}
As analyzed in subsection \ref{motivation}, the regularization parameter trades off between long and short step size. How can we reasonably choose this trade-off amount? To address this issue, at the current iterate, we consider two factors: the local morphology of the objective function's level set and search direction. 

\subsubsection{Local morphology of the objective function}\label{subsec:localmean}
\begin{definition}\cite{Armijo1966Minimizationfunctionshaving}
	Assume that for a given scalar $\mathcal{M}$ the function $f$ is continuously differentiable on the bounded level set $\mathcal{L}=\{\mathbf{x}:f(\mathbf{x})\le \mathcal{M}\}$. Gradient $\nabla f$ is called Lipschitz continuous on $\mathcal{L}$, if there exists a Lipschitz constant $L>0$, such that
	\begin{equation}\label{LP}
		\|\nabla f(\mathbf{x})-\nabla f(\mathbf{y})\|_{2}\le L\|\mathbf{x}-\mathbf{y}\|_{2},
	\end{equation}
	for every pair $\mathbf{x},\mathbf{y}\in\mathcal{L}$.
\end{definition}

In gradient descent methods, the step size $\frac{1}{\alpha_{k}}$ is related to the Lipschitz constant. \cite{Cauchy2009Methodegeneralepour} stated that the sequence $\{\mathbf{x}_{k}\}_{k=1}^{\infty}$ generated by the iterative scheme
\begin{equation}\label{equ:LP}
	\mathbf{x}_{k+1}=\mathbf{x}_{k}-\frac{1}{2L}\mathbf{g}_{k},\quad k=1,2,\ldots,
\end{equation}
is convergent. Notably, when the objective function is ``steep" (i.e., the Lipschitz constant $L$ is large), the iteration \eqref{equ:LP} requires a small step size to guarantee convergence. Conversely, when the function is ``flat" (i.e., $L$ is small), a large step size can be taken to accelerate convergence \cite{Vrahatis2000classgradientunconstrained}. However, for general functions, we neither know the function's morphology nor have prior knowledge of $L$. Fortunately, a local estimate $\Lambda_{k}$ for $L$ can be easily derived from \eqref{LP}, defined by 
\begin{equation}\label{equ:LLPP}
	\Lambda_{k}=\frac{\|\mathbf{g}_{k}-\mathbf{g}_{k-1}\|_{2}}{\|\mathbf{x}_{k}-\mathbf{x}_{k-1}\|_{2}}.
\end{equation}
This approach ensures that $\Lambda_{k}$ adapts to the local morphology of the objective function. From a computational perspective, $\Lambda_{k}$ in \eqref{equ:LLPP} corresponds precisely to the geometric mean of $\alpha_{k}^{BB1}$ and $\alpha_{k}^{BB2}$ as follows
\begin{equation}\label{equ:meanBB}
	\Lambda_{k}=\sqrt{\alpha_{k}^{BB1}\alpha_{k}^{BB2}}.
\end{equation}
Overall, \eqref{equ:meanBB} captures the local average curvature information of the objective function and characterizes its local shape features. 

\subsubsection{Search direction}\label{subsec:searchdirection}
We now investigate the effect of search direction $-\mathbf{g}_{k-1}$ on the step size. Let 
\begin{equation}\label{costheta}
 \theta_{k}=\angle(\mathbf{g}_{k-1},\mathbf{A}\mathbf{g}_{k-1}). 
\end{equation}
One has
\begin{equation}
 \frac{\alpha_{k}^{BB1}}{\alpha_{k}^{BB2}}=\cos^{2}\theta_{k}. 
\end{equation} 
If $\mathbf{g}_{k-1}$ approaches an eigenvector of $\mathbf{A}$, then $\cos^{2}\theta_{k}$ approximates $1$, and using a long step size becomes more effective than a short one. Otherwise, selecting a short step size is reasonable. This is the philosophy behind the ABB \cite{Zhou2006GradientMethodsAdaptive} method as follows 
\begin{equation}\label{ABB}
\alpha_{k}^{ABB}=\begin{cases}
\alpha_{k}^{BB2},\quad \text{if}\quad  \frac{\alpha_{k}^{BB1}}{\alpha_{k}^{BB2}}<\eta,\\
\alpha_{k}^{BB1},\quad \text{otherwise},
\end{cases}
\end{equation}
where $\eta\in(0,1)$ provided by user. To understand how the search direction influences the choice of step size, we start by analyzing the principle of the ABB method.

Based on the results in \cite{Forsythe1968asymptoticdirectionsthes,Yuan2006newstepsizesteepest}, the behavior of gradient descent method for higher dimensional problems is essential the same as it for two dimensional problems. Therefore, we consider minimizing the quadratic convex function \eqref{pro:quaa} with  
\begin{equation*}
\mathbf{A}=\begin{bmatrix}
\lambda & 0\\
0 & 1
\end{bmatrix},
\end{equation*}
where $\lambda>1$. We denote $\mathbf{g}_{k-1}=(\mathbf{g}_{k-1}^{(1)}, \ \mathbf{g}_{k-1}^{(2)})^{\T}$ with $\mathbf{g}_{k-1}^{(i)}\neq 0$ for $i=1,2$, and let
\begin{equation*}
\epsilon=\frac{(\mathbf{g}_{k-1}^{(1)})^{2}}{(\mathbf{g}_{k-1}^{(2)})^{2}}.
\end{equation*}
Then we have 
\begin{equation}\label{ratio}
\frac{\alpha_{k}^{BB1}}{\alpha_{k}^{BB2}}=\frac{(\lambda\epsilon+1)^2}{(\epsilon+1)(\lambda^{2}\epsilon+1)}.
\end{equation} 
From \eqref{ratio}, given an $\eta\in(0,1)$, if
\begin{equation}\label{equ:leta}
\frac{\alpha_{k}^{BB1}}{\alpha_{k}^{BB2}}<\eta,
\end{equation}
after rearrangement, then we have 
\begin{equation}\label{epsilon}
\phi(\epsilon):=\lambda^{2}(1-\eta)\epsilon^{2}+\big[2\lambda-\eta(1+\lambda^{2})\big]\epsilon+1-\eta<0.
\end{equation}
Note that $\phi(\epsilon)$ is a quadratic function with respect to $\epsilon$. We now analyze the roots of the function $\phi(\epsilon)$. From the discriminant of the quadratic function, we can see that if 
\begin{equation}\label{condition}
\eta>\frac{4\lambda}{(1+\lambda)^{2}},
\end{equation}
then the equation $\phi(\epsilon)=0$ has two positive real roots $\epsilon_{1}$ and $\epsilon_{2}$, and let $\epsilon_{1}<\epsilon_{2}$, which implies 
\begin{equation}\label{solutions}
\epsilon_{1}=\frac{\eta(1+\lambda^{2})-2\lambda-(\lambda^2-1)\sqrt{\eta\big(\eta-\frac{4\lambda}{(1+\lambda)^2}\big)}}{2\lambda^{2}(1-\eta)},\quad \epsilon_{2}=\frac{\eta(1+\lambda^{2})-2\lambda+(\lambda^2-1)\sqrt{\eta\big(\eta-\frac{4\lambda}{(1+\lambda)^2}\big)}}{2\lambda^{2}(1-\eta)}.
\end{equation}
From \eqref{solutions}, we obtain the  following two properties 
\begin{equation}\label{limita1a2}
\epsilon_{1}\to\frac{1-\eta}{\eta\lambda^{2}}=0,\quad \epsilon_{2}\to\frac{\eta}{1-\eta},\quad \text{as}\quad\lambda\to\infty.
\end{equation}
Note that if we further require $\eta<0.5$, which implies $\frac{\eta}{1-\eta}<1$. It follows from the results in \eqref{limita1a2} that 
\begin{equation}\label{results}
\begin{split} &{\lim_{\epsilon\to\epsilon_{2}}\alpha_{k}^{BB1}=\frac{\lambda\frac{\eta}{1-\eta}+1}{\frac{\eta}{1-\eta}+1}=\eta\lambda},\quad{\lim_{\epsilon\to\epsilon_{2}}\alpha_{k}^{BB2}=\frac{\lambda^2\frac{\eta}{1-\eta}+1}{\lambda\frac{\eta}{1-\eta}+1}=\lambda},\\ &{\lim_{\epsilon\to\epsilon_{1}}\alpha_{k}^{BB1}=\frac{\lambda\frac{1-\eta}{\eta\lambda^{2}}+1}{\frac{1-\eta}{\eta\lambda^{2}}+1}=1},\quad{\lim_{\epsilon\to\epsilon_{1}}\alpha_{k}^{BB2}=\frac{\lambda^2\frac{1-\eta}{\eta\lambda^{2}}+1}{\lambda\frac{1-\eta}{\eta\lambda^{2}}+1}=\frac{1}{\eta}}<\lambda (\text{due to \eqref{condition}}).
\end{split}
\end{equation}
If the inequality \eqref{equ:leta} holds, it implies that $\epsilon\in(\epsilon_{1}, \ \epsilon_{2})$, indicating that $\mathbf{g}_{k-1}^{(1)}$ has not been completely eliminated. In this case, $\alpha_{k}^{BB2}$ can better approximate the eigenvalue $\lambda$ corresponding to $\mathbf{g}_{k-1}^{(1)}$, while $\alpha_{k}^{BB1}$ can only approximate up to the $\eta\lambda$ level. If the inequality \eqref{equ:leta} does not hold, we know that either $\epsilon\rightarrow 0$ or $\epsilon\ge\frac{\eta}{1-\eta}$. In the former case, according to the results in \eqref{results}, $\alpha_{k}^{BB1}$ can better approximate the eigenvalue $1$ corresponding to $\mathbf{g}_{k-1}^{(2)}$. In the latter scenario, where the proportion of $|\mathbf{g}_{k-1}^{(1)}|$ is comparable to or exceeds that of $|\mathbf{g}_{k-1}^{(2)}|$, $\alpha_{k}^{BB1}$ can asymptotically reduce the proportion of $|\mathbf{g}_{k}^{(1)}|$ until condition \eqref{epsilon} is satisfied. These results thoroughly explain the principle of the ABB method. Additionally, by assigning $\lambda$ as infinity value here, we obtain favorable results. Nevertheless, in practice, $\lambda$ is a finite value. In this scenario, when $\epsilon\in(\epsilon_{1}, \ \epsilon_{2})$, $\alpha_{k}^{BB2}$ might no longer effectively approximate $\lambda$. A feasible strategy is to record the maximum value of $\alpha_{k}^{BB2}$ over successive iterations and use this maximum value to approximate $\lambda$. This also explains the mathematical principle behind ABBmin \cite{Frassoldati2008Newadaptivestepsize}.   

\subsection{A three-step regularization parameter scheme} 
Subsection \ref{subsec:localmean} employs the geometric mean of $\alpha_{k}^{BB1}$ and $\alpha_{k}^{BB2}$ to characterize the local mean curvature of the objective function, while subsection \ref{subsec:searchdirection} explains the influence of search direction on the step size. In this section, we integrate these two components to design the regularization parameter. 

We now  consider three scalars $\alpha_{k-1}^{BB2}$,  $\alpha_{k}^{BB2}$, and $\alpha_{k}^{BB1}$ and take  
\begin{equation}\label{REtau}
\tau_{k}=\frac{\alpha_{k}^{BB2}}{\alpha_{k}^{BB1}}\Big(\frac{\alpha_{k}^{BB2}}{\alpha_{k-1}^{BB2}}\Big)^2.
\end{equation}  
We call \eqref{REtau} a three-step regularization parameter. Then, we have
\begin{equation}\label{extau}
\tau_{k}=\frac{\alpha_{k}^{BB2}}{\alpha_{k}^{BB1}}\Big(\frac{\alpha_{k}^{BB2}}{\alpha_{k-1}^{BB2}}\Big)^2=\frac{1}{\cos^2\theta_{k}}\frac{\|\mathbf{y}_{k-1}\|_{2}^{2}/\|\mathbf{s}_{k-1}\|_{2}^{2}}{\|\mathbf{y}_{k-2}\|_{2}^{2}/\|\mathbf{s}_{k-2}\|_{2}^{2}}\frac{\cos^2(\theta_{k-1})}{\cos^2(\theta_{k})}=\frac{\Lambda_{k}^2}{\Lambda_{k-1}^2}\zeta_{k},
\end{equation}
where
\begin{equation}\label{meanValue}
\zeta_{k}=\frac{1}{\cos^2\theta_{k}}\frac{\cos^2(\theta_{k-1})}{\cos^2(\theta_{k})}.
\end{equation}
The ratio $\frac{\Lambda_{k}^2}{\Lambda_{k-1}^2}$ in \eqref{extau} describes the relative change in the local mean curvature between two consecutive iterations. If this ratio exceeds $1$, it indicates that the function's morphology at the current iterate is ``steep", choosing a short step size is  appropriate; conversely, selecting a long step size is reasonable.

We now focus on the analysis of $\zeta_{k}$ in \eqref{extau}. If $\cos^{2}\theta_{k}\rightarrow1$ and $\cos^{2}\theta_{k-1}\rightarrow0$, then $\frac{\cos^{2}\theta_{k-1}}{\cos^{2}\theta_{k}}\rightarrow0$ (i.e., selecting $\alpha_{k}^{BB1}$). If $\cos^{2}\theta_{k}\rightarrow0$ and $\cos^{2}\theta_{k-1}\rightarrow1$, then $\frac{\cos^{2}\theta_{k-1}}{\cos^{2}\theta_{k}}\rightarrow\infty$ (i.e., selecting $\alpha_{k}^{BB2}$). Specifically, we analyze the principle of $\zeta_{k}$ from the following three cases:
\begin{enumerate}
	\item[(1)] $\frac{\cos^2\theta_{k-1}}{\cos^2\theta_{k}}<\cos^2\theta_{k}<1\iff\cos^2\theta_{k-1}<\cos^4\theta_{k}\Rightarrow 0<\zeta_{k}<1$, 
	\item[(2)] $\cos^2\theta_{k}<\frac{\cos^2\theta_{k-1}}{\cos^2\theta_{k}}<1\iff\cos^2\theta_{k-1}>\cos^4\theta_{k}\Rightarrow \zeta_{k}>1$,
	\item[(3)] $\cos^2\theta_{k}<1<\frac{\cos^2\theta_{k-1}}{\cos^2\theta_{k}}\iff\cos^2\theta_{k-1}>\cos^4\theta_{k}\Rightarrow \zeta_{k}>1$.
\end{enumerate}
The first case implies that the $\cos^{2}\theta_{k}$ is large, making it advantageous to choose a long step size, while the latter two scenarios correspond to small $\cos^{2}\theta_{k}$ where employing a short step size is appropriate.

In summary, at the current iterate, a large regularization parameter $\tau_{k}$ corresponds to a ``steep" function (i.e., $\frac{\Lambda_{k}^2}{\Lambda_{k-1}^2}>1$) and a search direction $-\mathbf{g}_{k-1}$ that deviates  from an eigenvector of $\mathbf{A}$ (i.e.,  $\zeta_{k}>1$), and leads to a short step size. A small regularization parameter $\tau_{k}$ corresponds to a ``flat" function (i.e., $\frac{\Lambda_{k}^2}{\Lambda_{k-1}^2}<1$) and a search direction $-\mathbf{g}_{k-1}$ aligned with an eigenvector of $\mathbf{A}$ (i.e., $\zeta_{k}<1$), and introduces a long step size. 

However, the $\tau_{k}$ in \eqref{REtau} only provides fundamental guidance for selecting long or short step size, as it cannot rapidly approach $0$ or $\infty$. In practice, we need to adjust $\tau_{k}$: when it is greater than $1$, the adjusted $\tau_{k}$ must swiftly approach $\infty$; when it is less than $1$, the adjusted $\tau_{k}$ must swiftly approach $0$. A feasible strategy is to employ a scaling operation of the following form:
\begin{equation}\label{Retauk}
\tau_{k}(q)=(\tau_{k})^{q},
\end{equation}
for $q\ge1$ is a scaling factor. We will test various choice for $q$ in Section \ref{sec:numerical}.


\section{An enhanced RBB (ERBB) method}\label{sec:ERBB}
Reviewing the analysis of the ABB method in subsection \ref{subsec:searchdirection}, we pointed out that if the inequality \eqref{equ:leta} does not hold, i.e., $\frac{\alpha_{k}^{BB1}}{\alpha_{k}^{BB2}}\ge\eta$, two cases arise: (1) $\epsilon\rightarrow 0$, and (2) $\epsilon\ge\frac{\eta}{1-\eta}$. For the first case, selecting $\alpha_{k}^{BB1}$ is effective. The second case implies that the proportion of $|\mathbf{g}_{k-1}^{(1)}|$ is comparable to or exceeds that of $|\mathbf{g}_{k-1}^{(2)}|$, in which case selecting a short step size like BB2 is more appropriate. Nevertheless, the original ABB method only accounts for the first case.

The key lies in distinguishing between the two cases. We continue to study the two-dimensional problem in subsection \ref{subsec:searchdirection}. Since $\alpha_{k}^{BB1}=\frac{\epsilon\lambda+1}{\epsilon+1}$, when $\epsilon$ approaches $0$, $\alpha_{k}^{BB1}$ approximates $1$, and a corresponding result is that $\alpha_{k}^{BB1}<\alpha_{k-1}^{BB1}$ or $\alpha_{k}^{BB1}<\alpha_{k-1}^{BB2}$.   Based on this observation, in order to identify the differences between the two as accurately as possible, if $\alpha_{k}^{BB1}>\alpha_{k-1}^{BB2}$, we conclude that $\epsilon\ge\frac{\eta}{1-\eta}$. This forms the second criterion for step size selection, refining the choice of alternating step size.  

We incorporate the second criterion into  the ABB method and derive an ERBB method by employing the delay strategy of the ABBmin method as follows 
\begin{equation}\label{alternate step}
\alpha_{k}^{ERBB}=\begin{cases}
\max\big\{\alpha_{j}^{RBB}|j\in\{j_{0},\,\ldots,\,k\}\big\}, & \text{if}\quad \cos^2\theta_{k}<\mu_{k},\\
\max\big\{\alpha_{k}^{BB2},\alpha_{k-1}^{BB2}\big\}, & \text{elseif}\quad \cos^2\theta_{k}\ge\mu_{k} \ \text{and} \ \alpha_{k}^{BB1}>\alpha_{k-1}^{BB2},\\ 
\alpha_{k}^{BB1},& \text{otherwise},
\end{cases}
\end{equation} 
where $j_{0}=\max\{1,\,k-\varrho\}$, $\varrho\ge0$ is an integer,
\begin{equation*}\label{nuk}
\mu_{k}=1-\frac{\alpha_{k}^{BB1}}{\alpha_{k}^{RBB}}.	
\end{equation*}
It follows from the results of Theorem \ref{montonic} that $\alpha_{k}^{BB1}\le\alpha_{k}^{RBB}\le\alpha_{k}^{BB2}$, and thus
\begin{equation}\label{le}
	\frac{\alpha_{k}^{BB1}}{\alpha_{k}^{BB2	}}\le\frac{\alpha_{k}^{BB1}}{\alpha_{k}^{RBB}}.	
\end{equation}
If $\frac{\alpha_{k}^{BB1}}{\alpha_{k}^{BB2}}<\mu_{k}$, then we have $\frac{\alpha_{k}^{BB1}}{\alpha_{k}^{BB2}}<0.5$ due to \eqref{le}. In this case, according to the \text{ABB's} principle, it is appropriate to choose a short step size.

Note that criterion $\frac{\alpha_{k}^{BB1}}{\alpha_{k}^{BB2}}<\mu_{k}$ is not equivalent to $\frac{\alpha_{k}^{BB1}}{\alpha_{k}^{BB2}}<0.5$. For example, if $\frac{\alpha_{k}^{BB1}}{\alpha_{k}^{RBB}}=0.8$, this indicates a weak regularization, which may result from either a small local mean curvature of the objective function or the search direction $-\mathbf{g}_{k-1}$ aligning closely with an eigenvector of the Hessian. In such case, a short step size can only be selected if $\frac{\alpha_{k}^{BB1}}{\alpha_{k}^{BB2}}<0.2$. This implies that the low regularization stems not from $-\mathbf{g}_{k-1}$ but from the local mean curvature of objective function. In other words, $0.5$ is the upper limit for $\mu_{k}$.

For the strictly convex quadratic problems \eqref{pro:quaa}, the $R$-linear  convergence of the RBB, RBBA \eqref{equ: quaregBBstepsize} and ERBB methods \eqref{alternate step} can be easily established using the results in \cite{Dai2003Alternatestepgradient,Li2023NoteRLinear}. 

At first glance, the computational complexity of $\alpha_{k}^{RBB}$ \eqref{equ: quaregBBstepsize} and $\alpha_{k}^{ERBB}$ \eqref{alternate step} may seem high, but in reality, it is nearly identical to that of $\alpha_{k}^{ABB}$, as demonstrated \cite{Zhou2006GradientMethodsAdaptive} below: 
\begin{equation*}{\label{simple}
	\begin{cases}
	&\mathbf{s}_{k-1}^{\T}\mathbf{s}_{k-1}=t_{k-1}^{2}\mathbf{g}_{k-1}^{\T}\mathbf{g}_{k-1}, \\
	&\mathbf{s}_{k-1}^{\T}\mathbf{y}_{k-1}=t_{k-1}\mathbf{g}_{k-1}^{\T}(\mathbf{g}_{k-1}-\mathbf{g}_{k}), \\ 
	&\mathbf{y}_{k-1}^{\T}\mathbf{y}_{k-1}=\mathbf{g}_{k}^{\T}\mathbf{g}_{k}-2\mathbf{g}_{k-1}^{\T}\mathbf{g}_{k}+\mathbf{g}_{k-1}^{\T}\mathbf{g}_{k-1},
	\end{cases}
}\end{equation*} 
where $t_{k-1}=\frac{1}{\alpha_{k-1}^{RBB}}$ or $t_{k-1}=\frac{1}{\alpha_{k-1}^{ERBB}}$. Notice that at every iteration, we only need to compute two inner products, i.e., $\mathbf{g}_{k-1}^{\T}\mathbf{g}_{k}$ and $\mathbf{g}_{k}^{\T}\mathbf{g}_{k}$. This indicates that their computational complexity is comparable to that of  $\alpha_{k}^{BB2}$, and they have an additional inner product operation than $\alpha_{k}^{BB1}$. The regularization parameter \eqref{Retauk} involves only scalar operations. Compared to the RBB method, ERBB stores a few additional scalars, and these computation costs remain negligible. Nevertheless, the RBBA method requires the computation of $n+3$ inner products at each iteration.

\begin{remark}
	The RBBA method exhibits well numerical performance in quadratic problems, as demonstrated in the experimental section. A method to reduce its complexity involves considering the following model 
	\begin{equation*}
	\min_{\alpha>0} \Big\{\Vert \alpha  \mathbf{s}_{k-1} - \mathbf{y}_{k-1}\Vert_{2}^{2} + \tau_{k}\Vert  \alpha \mathbf{y}_{k-1} - \Phi_{k} \mathbf{y}_{k-1}\Vert_{2}^{2} \Big\},
	\end{equation*}
	where $\Phi_{k}\ge\alpha_{k}^{BB2}$. Further improvements in this direction, nevertheless, fall outside the scope of this paper.
\end{remark}

\subsection{\text{Non-quadratic} minimization}\label{sec:nonqua}
In order to extend the \text{RBB} and \text{ERBB} methods for minimizing \text{non-quadratic} continuous differentiable functions \eqref{generalqua}, we usually need to incorporate some line search strategies to ensure global convergence. When $f$ is a generic function, the average Hessian $\mathbf{A}_{k}=\int_{0}^{1}\nabla^{2}f(\mathbf{x}_{k-1}+t\mathbf{s}_{k-1})dt$ satisfies the secant equation $\mathbf{y}_{k-1}=\mathbf{A}_{k}\mathbf{s}_{k-1}$ (cf. e.g., \cite[Eq.(6.11)]{JorgeNocedal2006NumericalOptimization}), and thus we can still view the $\alpha_{k}^{RBB}$ and $\alpha_{k}^{ERBB}$ as Rayleigh quotients, which approximate the eigenvalues of this average $\mathbf{A}_{k}$. We note that under the condition that $\mathbf{A}_{k}$ is \text{SPD}, all results in preceding sections are still valid for generic functions with replacing $\mathbf{A}$ by $\mathbf{A}_{k}$.  

Among BB-like methods, \text{nonmonotonic} line search is an effective strategy \cite{Raydan1997BarzilaiBorweinGradienta}. Here we would like to adopt the \text{Grippo-Lampariello-Lucidi (GLL) nonmonotonic} line search \cite{Grippo1986NonmonotoneLineSearch}, which accepts $\gamma_{k}\in(0,1)$ when it satisfies
\begin{equation}\label{non-monotone}
f(\mathbf{x}_{k}+\gamma_{k}\mathbf{d}_{k})\le\max_{1\le j\le \min\{k,M\}} \big\{f(\mathbf{x}_{k-j+1})\big\} +\sigma\gamma_{k}\mathbf{g}_{k}^{\T}\mathbf{d}_{k},	
\end{equation}  
where $M$ is a nonnegative integer, $\sigma\in(0,\,1)$, and  $\mathbf{d}_{k}=-\frac{1}{\alpha_{k}}\mathbf{g}_{k}$. At the start of each internal line search, we set $\gamma_{k}=1$. To improve the efficiency of the line search, we add a quadratic interpolation strategy \cite[p.58]{JorgeNocedal2006NumericalOptimization} to the GLL line search using the values of $f(\mathbf{x}_{k})$, $f(\mathbf{x}_{k}+\mathbf{d}_{k})$ and $\mathbf{g}_{k}^{\T}\mathbf{d}_{k}$ and obtain the following scalar
\begin{equation}\label{interpolation}
	\bar{\gamma}_{k}=\frac{-\mathbf{g}_{k}^{\T}\mathbf{d}_{k}\gamma_{k}^2}{2\Big(f(\mathbf{x}_{k}+\mathbf{d}_{k})-f(\mathbf{x}_{k})-\gamma_{k}\mathbf{g}_{k}^{\T}\mathbf{d}_{k}\Big)}.
\end{equation}
The detailed procedure is presented in Algorithm \ref{alg:RBB}. Line $2$ describes a condition for the function value to decrease sufficiently, where the common values of line search parameters $\delta=\frac{1}{2}$, $\sigma=10^{-4}$ (cf. \cite[p.33]{JorgeNocedal2006NumericalOptimization}), and $M=10$ in the experiments. The critical assumption to prove the global convergence of Algorithm \ref{alg:RBB} is that step size $\frac{1}{\alpha_{k}}$ is uniformly bounded, i.e., $\alpha_{k}\in[\alpha_{\min}, \alpha_{\max}]$ for all $k$. Since $\frac{1}{\alpha_{k}^{RBB}}$ \eqref{equ: quaregBBstepsize} or $\frac{1}{\alpha_{k}^{ERBB}}$ \eqref{alternate step} with safeguard lies in this interval, the convergence of Algorithm \ref{alg:RBB} is guaranteed by \cite[Thm.2.1]{Raydan1997BarzilaiBorweinGradienta}. The $R$-linear convergence of Algorithm \ref{alg:RBB} can be proved for uniformly convex functions \cite{Dai2002NonmonotoneLineSearch}.

\begin{algorithm}
	\caption{Regularized \BB algorithm for solving general unconstrained optimization problems}\label{alg:RBB}
	\begin{algorithmic}[1]
		\Require Stopping criterion: $\varepsilon>0$, ${\text{MaxIt}>1}$;
		Initialization: $\mathbf{x}_{1}$, $\alpha_{1}\in [\alpha_{\text{min}},\alpha_{\text{max}}]$, $\sigma$, $\delta \in(0,1)$, $M$, $q$, $k=1$, $\mathbf{d}_{k}=-\frac{1}{\alpha_{k}}\mathbf{g}_{k}$, $\gamma_{k}=1$, $\tau_{1}$.
		\While{$\|\mathbf{g}_{k}\|_{2}>\varepsilon$ \rm{or} $k<\rm{MaxIt}$}{ 
		\If{$f(\mathbf{x}_{k}+\gamma_{k}\mathbf{d}_{k})\le\max_{1\le j\le \min\{k,\,M\}} \big\{f(\mathbf{x}_{k-j+1})\big\} +\sigma\gamma_{k}\mathbf{g}_{k}^{\T}\mathbf{d}_{k}$}
		\State 
		 	$\mathbf{x}_{k+1}=\mathbf{x}_{k}+\gamma_{k}\mathbf{d}_{k}$
		 	\If{$\mathbf{s}_{k}^{\T}\mathbf{y}_{k}<0$}
		 	\State
		 	$\bar{\alpha}_{k+1}=\hat{\alpha}_{k+1}$
		 	\Else
		 	\State 
		 	calculate $\bar{\alpha}_{k+1}=\alpha_{k+1}^{RBB}$ \eqref{equ: quaregBBstepsize} \big(\text{or}  $\bar{\alpha}_{k+1}=\alpha_{k+1}^{ERBB}$ \eqref{alternate step}\big)
		 	\EndIf\\
		 		\;\;\qquad\qquad 
		 		set $\alpha_{k+1}=\min\big\{\max\{\bar{\alpha}_{k+1}, \alpha_{\min}\}, \alpha_{\max}\big\}$\\
		 		\;\;\qquad\qquad 
		 		$\mathbf{d}_{k+1}=-\frac{1}{\alpha_{k+1}}\mathbf{g}_{k+1}$, $\gamma_{k+1}=1$\\
		 	\;\;\qquad\qquad set $k=k+1$  
		\Else 
			\If{$\gamma_{k}\le0.1$}
		 	\State $\gamma_{k}=\delta\gamma_{k}$
		 	\Else
		 	\State
		 	calculate $\bar{\gamma}_{k}$ by \eqref{interpolation}
		 		\If{$\bar{\gamma}_{k}<0.1$\quad\rm{or}\quad $\bar{\gamma}_{k}>0.9\gamma_{k}$}
		 		\State
		 		$\gamma_{k}=\delta\bar{\gamma}_{k}$
		 		\Else
		 		\State
		 		$\gamma_{k}=\bar{\gamma}_{k}$
		 		\EndIf
		 	\EndIf
		\EndIf
			}
		\EndWhile
	\end{algorithmic}
\end{algorithm}

The initial step size and the treatment of uphill direction ($\mathbf{s}_{k-1}^{\T}\mathbf{y}_{k-1}<0$, i.e., $\mathbf{A}_{k}$ is not \text{SPD}) are two important factors that affect the performance of Algorithm \ref{alg:RBB}. Popular choice for the initial step size are $\frac{1}{\alpha_{1}}=1$ (cf. eg., \cite{Raydan1997BarzilaiBorweinGradienta,Serafino2018steplengthselectiongradient}) or $\frac{1}{\alpha_{1}}=\frac{1}{\|\mathbf{g}_{1}\|}$ (cf. eg., \cite{Dai2003Alternatestepgradient}), where the norm is the Euclidean $2$-norm or $\infty$-norm. If $\mathbf{s}_{k-1}^{\T}\mathbf{y}_{k-1}<0$, $\alpha_{k}^{RBB}$ and $\alpha_{k}^{ERBB}$ may be negative. In this case, the tentative $\alpha_{k}$ is replaced by a certain $\hat{\alpha}_{k}>0$. A possible choice is $\frac{1}{\hat{\alpha}_{k}}=\frac{\|\mathbf{s}_{k-1}\|_2}{\|\mathbf{y}_{k-1}\|_{2}}$ \cite{Burdakov2019StabilizedBarzilaiBorwein}. \text{Raydan} \cite{Raydan1997BarzilaiBorweinGradienta} suggests using $\frac{1}{\hat{\alpha}_{k}}=\max\big(\min(\|\mathbf{g}_{k}\|_{2}^{-1}, 10^{5}),1\big)$, which makes the sequence $\{\frac{1}{\alpha_{k}}\}$ remains uniformly bounded while keeping $\|\frac{1}{\hat{\alpha}_{k}}\mathbf{g}_{k}\|_{2}$ moderate. Some authors use $\frac{1}{\hat{\alpha}_{k}}=\|\mathbf{g}_{k}\|_{2}^{-1}$, like the initial step size setting, similar to the restart operation.

\section{Numerical experiments}\label{sec:numerical}
In this section, we conduct preliminary numerical experiments \footnote{All experiments were implemented in \text{MATLAB R2024a}. All the runs were carried out on a PC with an 12th Gen Intel(R) Core(TM) i7-12700H 2.30 GHz and 32 GB of RAM} to illustrate the performance of the \text{RBB}, \text{RBBA}, and \text{ERBB} methods. We first report the numerical performance of the RBB step size $\frac{1}{\alpha_{k}^{RBB}}$ \eqref{equ: quaregBBstepsize} with various $q$ in \eqref{Retauk}, then conduct numerical comparison experiments with some outstanding algorithms of the same type. 
\subsection{Choice of q}\label{subsec:choice q}
In this subsection, we evaluate the performance of various $q$ values using the performance profile \cite{Dolan2002Benchmarkingoptimizationsoftware}. The cost of solving each problem is normalized according to the lowest cost of solving that problem to obtain the performance ratio $\tau$. The most efficient method solves a given problem with a performance ratio $1$, while all other methods solve the problem with a performance ratio of at least $1$. In order to grasp the full implications of our test data regarding the solvers' probability of successfully handing a problem, we display a $log$ scale of the performance profiles. Since we are also interested in the behavior for $\tau$ close to $1$, we use use a base of $2$ for the scale \cite{Dolan2002Benchmarkingoptimizationsoftware}. Therefore, the value of $\rho_{s}(\log_2(1))$ is the probability that the solver $s$ will win over the rest of the solvers. Unless otherwise specified, the performance profiles mentioned in subsequent experiments are all $\log_2$ scaled. For convenience, we denote $\omega=\log_2(\tau)$ in this paper.

Consider the following quadratic function from \cite{DeAsmundis2014efficientgradientmethod}:
\begin{equation}\label{pro:quadra}
f(\mathbf{x})=\frac{1}{2}(\mathbf{x}-\mathbf{x}_{*})^{\T}\mathbf{A}(\mathbf{x}-\mathbf{x}_{*}),
\end{equation}
where $\mathbf{x}_{*}$ is uniformly and randomly generated from $[-10, 10]^{n}$, $\mathbf{A}=\mathbf{Q}\cdot\text{diag}(v_{1},\ldots,v_{n})\cdot \mathbf{Q}^{\T}$ with  $\mathbf{Q}=(\mathbf{I}-2\mathbf{\omega}_3\mathbf{\omega}_3^{\T})(\mathbf{I}-2\mathbf{\omega}_2\mathbf{\omega}_2^{\T})(\mathbf{I}-2\mathbf{\omega}_1\mathbf{\omega}_1^{\T})$,  $\mathbf{\omega}_1$, $\mathbf{\omega}_2$ and $\mathbf{\omega}_3$ being unit random vectors, $v_1=1$, $v_n=\kappa(\mathbf{A})$ and $v_j$ is randomly generated between $1$ and $\kappa(\mathbf{A})$ for $j=2,\ldots, n-1$. 

We set $n\in\{100, 1000\}$ and $\kappa(\mathbf{A})=10^3, 10^4, 10^5, 10^6$. The initial guess is a vector randomly generated from $[-5, 5]^{n}$. The stopping criterion is either the gradient at the $k$-th iteration satisfies that $\|\mathbf{g}_{k}\|_{2}\le\varepsilon\|\mathbf{g}_{1}\|_{2}$ with $\varepsilon=10^{-6}, 10^{-8}, 10^{-10}$ or the number of iterations exceeds $20000$. The initial step size is the steepest descent step $\frac{\mathbf{g}_{1}^{\T}\mathbf{g}_{1}}{\mathbf{g}_{1}^{\T}\mathbf{A}\mathbf{g}_{1}}$. The code was independently executed 20 times to investigate the impact of the scaling factor $q$ in the regularization parameters \eqref{Retauk} on the performance of the RBB method.

Figure \ref{fig:parameterRBB} displays the performance profile of the RBB method with different $q$ \eqref{Retauk} on the quadratic problems \eqref{pro:quadra}, based on the number of iterations. It is evident that the best choice of $q$ appear to be $4$, $6$ or $8$ for $n=100$, $2$ or $8$ for $n=1000$. In order to achieve a fast approximation to the BB1 or BB2 step, through this paper, we set $q=8$ in $\alpha_{k}^{RBB}(\tau_{k}(q))$. In this problem, compared to the BB1 and BB2 methods, the RBB method demonstrates a significant advantage even for $q=1$, indicating the effectiveness of the regularization parameter scheme \eqref{Retauk}. 
\begin{figure}[!ht]
	\centering
	\subfigure[n=100]{
		\includegraphics[width=0.48\textwidth]{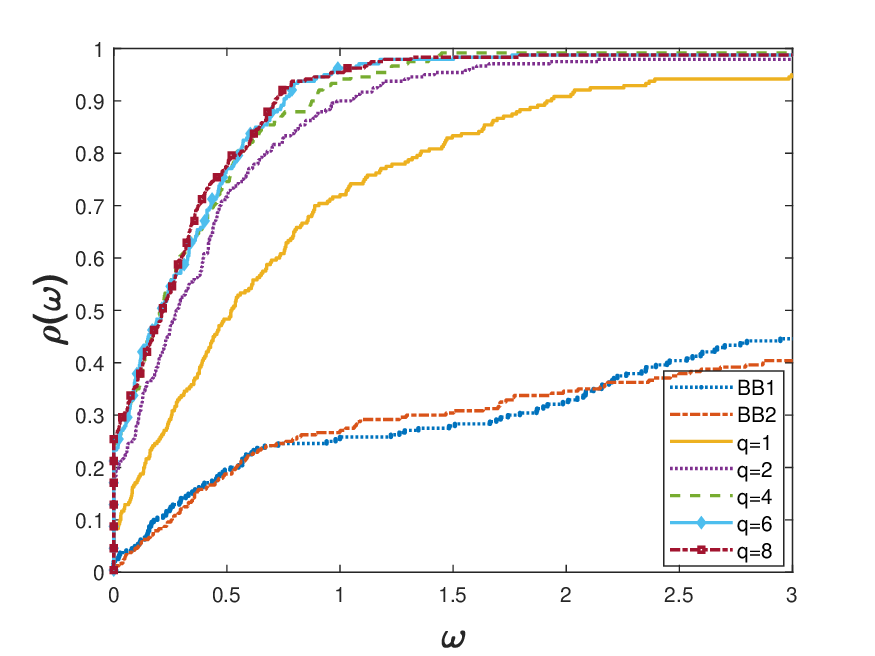}}\hspace{-4pt}
	\subfigure[n=1000]{
		\includegraphics[width=0.48\textwidth]{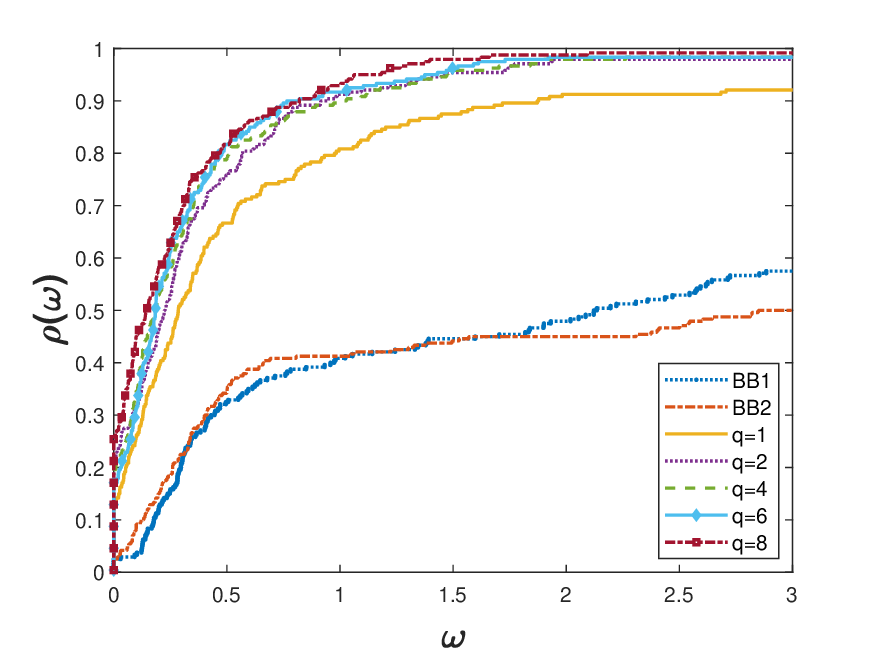}}\\
	\caption{\textit{Performance profiles of the RBB method with different $q$ on the quadratic problems \eqref{pro:quadra}, iteration metric}}	
	\label{fig:parameterRBB}
\end{figure}
\subsection{Test on the non-stochastic problem \eqref{diagconditionn}}
In this subsection, we investigate the behavior of the RBB, RBBA, and ERBB methods on the non-stochastic problem \eqref{diagconditionn}. We also consider the ABBmin method \cite{Frassoldati2008Newadaptivestepsize}
\begin{equation}\label{ABBmin}
\alpha_{k}^{ABBmin}=\begin{cases}
\max\big\{\alpha_{j}^{BB2}|j=\max\{1,k-m\},\ldots,k\big\},\quad&\text{if}\quad\cos^2\theta_{k}<\nu,\\
\alpha_{k}^{BB1},\quad&\text{otherwise},
\end{cases}
\end{equation} 
and the parameters $m$ and $\varrho$ are both set to $9$ in ABBmin and ERBB. We set $\nu=0.8$ in \eqref{ABBmin}. The problem-dependent parameters are the same as in subsection \ref{motivation}. Figure \ref{fig:Motivation2} and \ref{fig:Motivation3} record the history of $\alpha_{k}$ and the $\ell_{2}$-norm of the gradient generated by these methods.
\begin{figure}[!h]
	\centering
	\subfigure{
		\includegraphics[width=0.48\textwidth]{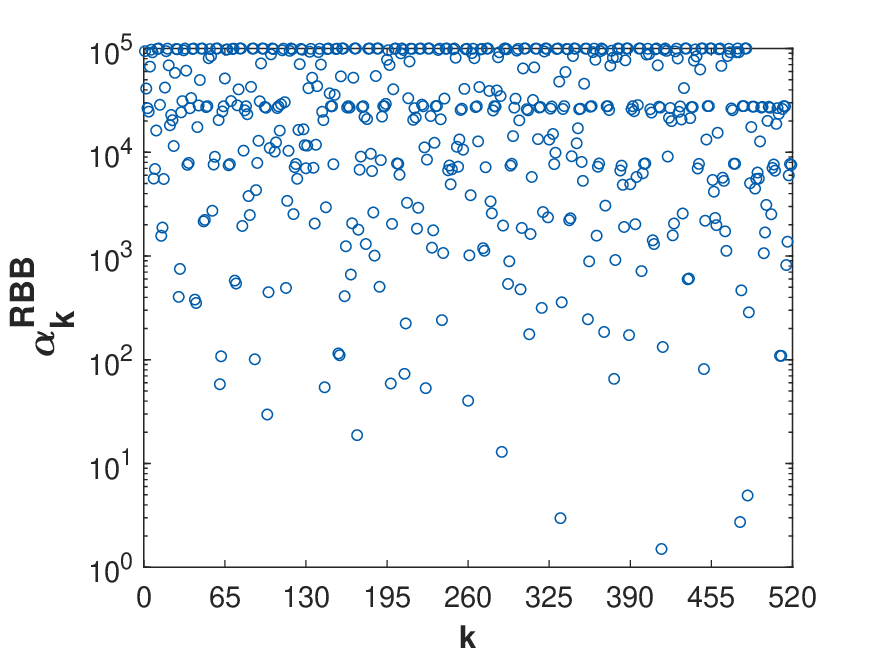}}\hspace{-4pt}
	\subfigure{
		\includegraphics[width=0.48\textwidth]{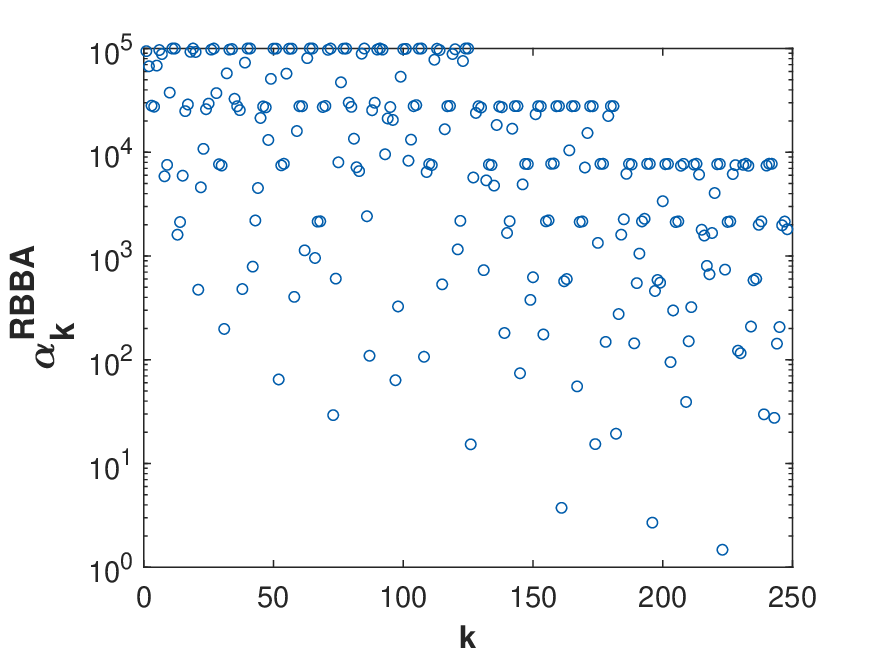}}\\
	\subfigure{
		\includegraphics[width=0.48\textwidth]{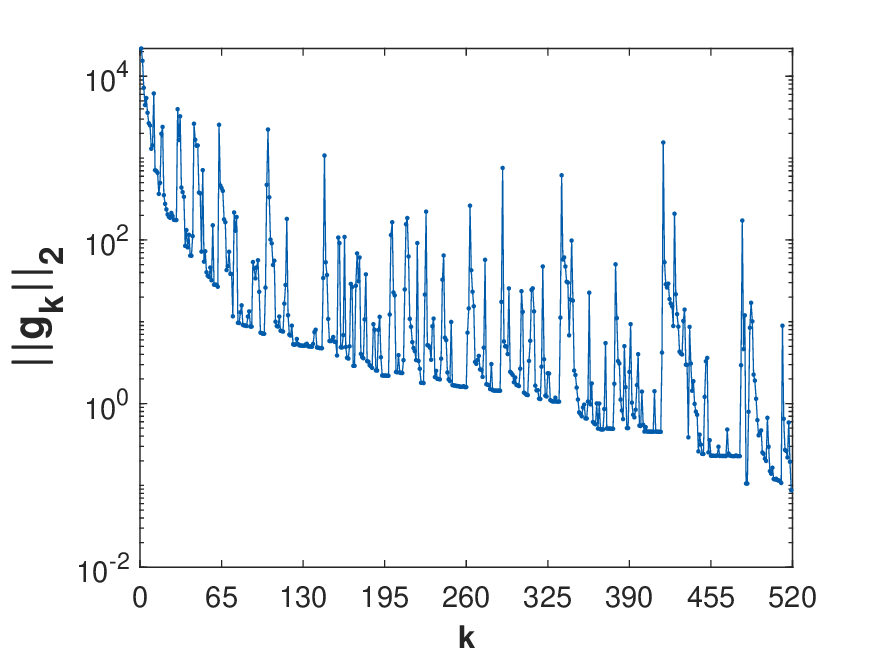}}\hspace{-4pt}
	\subfigure{
		\includegraphics[width=0.48\textwidth]{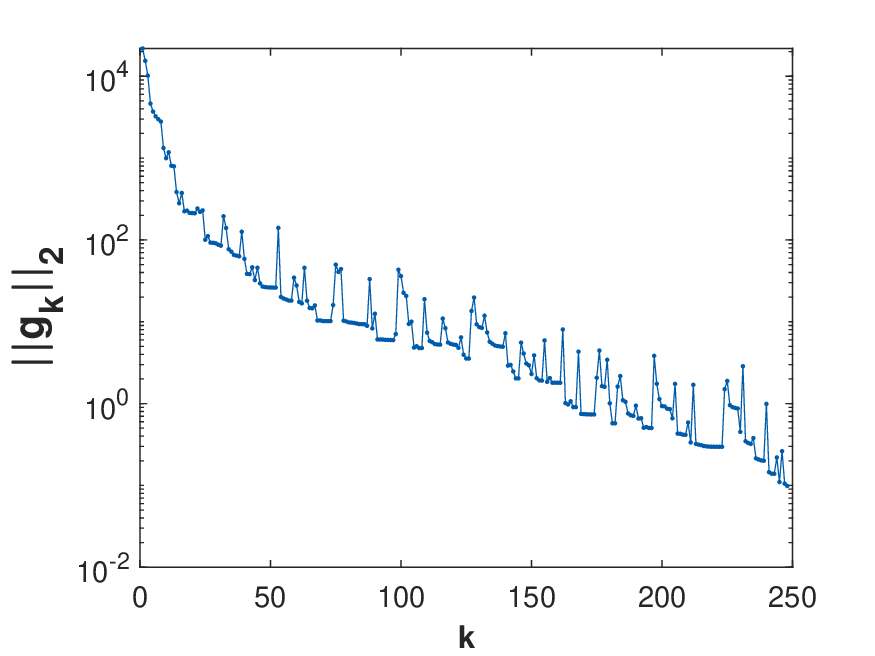}}\\
	\caption{\textit{Problem \eqref{diagconditionn} with $n=10$: historical values of $\alpha_{k}$ (top) and corresponding $\|\mathbf{g}_{k}\|_{2}$ (bottom) generated by the RBB and RBBA methods}}	
	\label{fig:Motivation2}
\end{figure}

\begin{figure}[!h]
	\centering
	\subfigure{
		\includegraphics[width=0.48\textwidth]{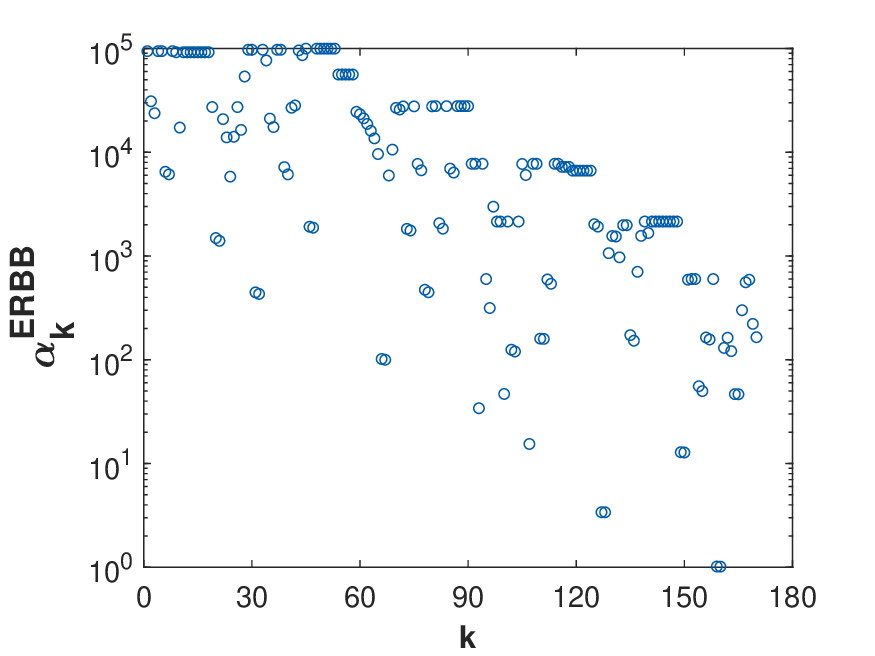}}\hspace{-4pt}
	\subfigure{
		\includegraphics[width=0.48\textwidth]{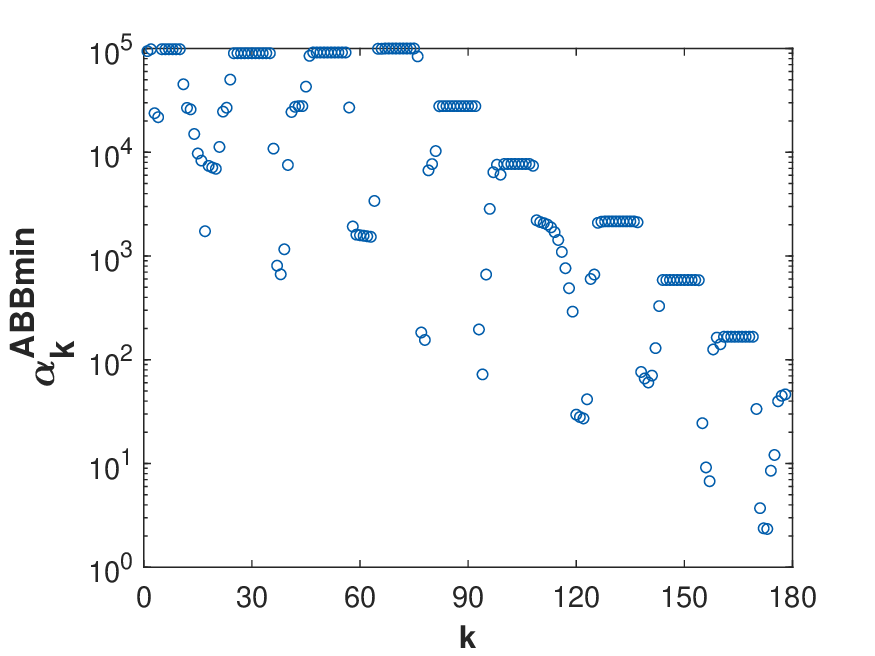}}\\
	\subfigure{
		\includegraphics[width=0.48\textwidth]{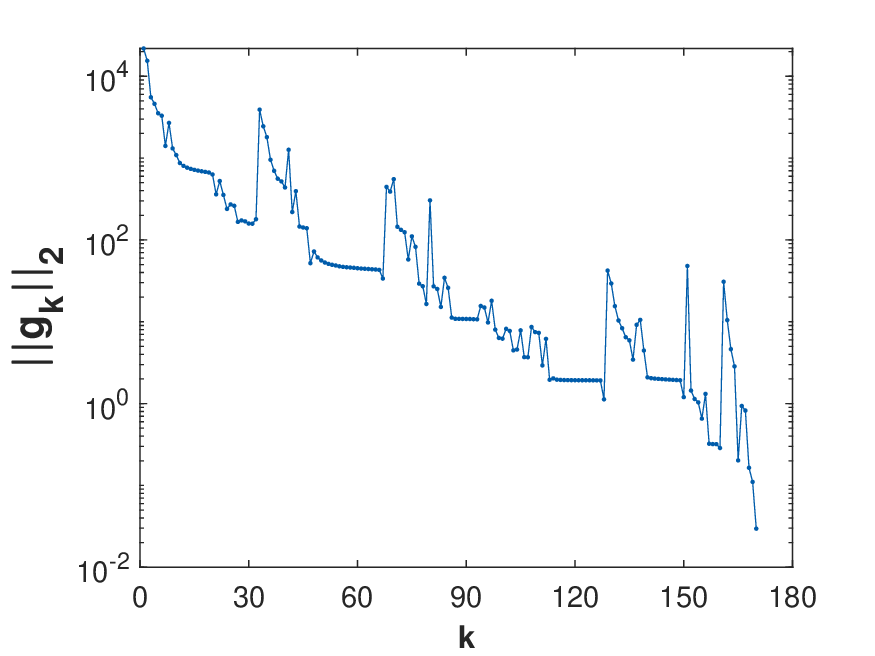}}\hspace{-4pt}
	\subfigure{
		\includegraphics[width=0.48\textwidth]{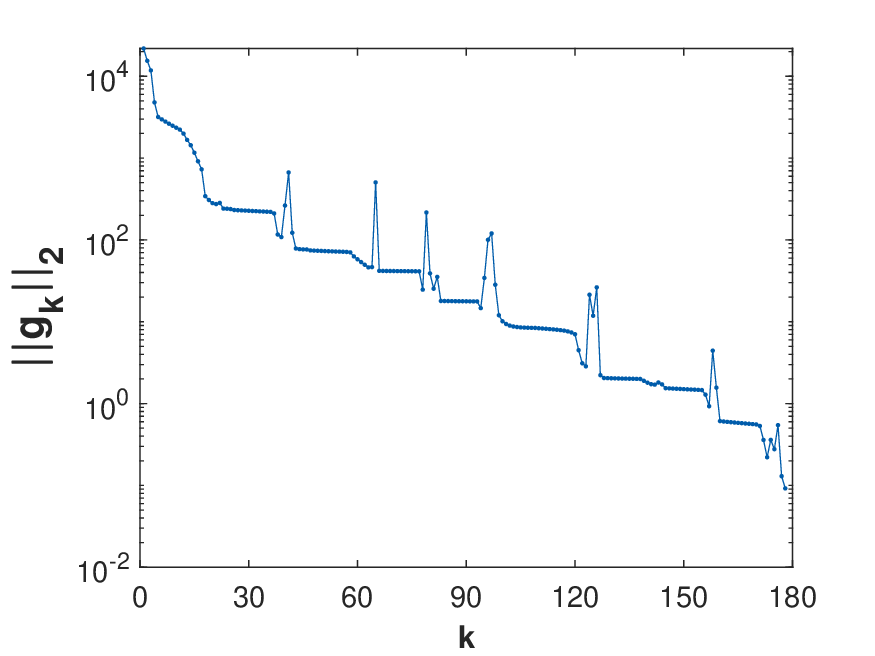}}\\
	\caption{\textit{Problem \eqref{diagconditionn} with $n=10$: historical values of $\alpha_{k}$ (top) and corresponding $\|\mathbf{g}_{k}\|_{2}$ (bottom) generated by the ERBB and ABBmin methods}}	
	\label{fig:Motivation3}
\end{figure}

From the results in Figure \ref{fig:Motivation2}, it can be observed that RBBA performs better than RBB, which is attributed to the fact that $\alpha_{k}^{RBBA}$ can more easily approximate the large eigenvalues of $\mathbf{A}$, making it more asymptotically convergent, that is, $\alpha_{k}^{RBBA}$ has a stronger ability to scan the spectrum of Hessian $\mathbf{A}$. In Figure \ref{fig:Motivation3}, observing the difference in the behavior of $\alpha_{k}^{ERBB}$ and $\alpha_{k}^{ABBmin}$, the former achieves the first global decrease at the $53$rd iteration and approaches the minimum eigenvalue $1$ of $\mathbf{A}$ at the $159$th iteration, while the latter achieves the first global decrease at the $76$th iteration, which indicates that the second alternation criterion in the ERBB method is effective, which improves the convergence of the alternating step size strategy. 
\subsection{Comparing with the outstanding algorithms for quadratics }
We test on the quadratic function \eqref{pro:quadra} with seven kinds of distributions of $v_{j}$ summarized in Table \ref{tab:spectrum} from \cite{Dai2019familyspectralgradient}. We set the dimension $n=1000$, $\varepsilon=10^{-6}, 10^{-10}$ corresponding to the low-precision and high-precision convergence respectively, while other configurations including the initial point and condition number $\kappa(\mathbf{A})$ remain identical to those in subsection \ref{subsec:choice q}. 
\begin{table}[!ht]
	\centering
	\caption{Different distributions of $v_{j}$ for the problem \eqref{pro:quadra}}
	\setlength{\tabcolsep}{10pt}{
		\begin{tabular}{|cccc|cccc|}
			\toprule
			P     & \multicolumn{3}{c|}{$v_{j}$} & P     & \multicolumn{3}{c|}{$v_{j}$} \\
			\hline
			1     & \multicolumn{3}{c|}{$\{v_{2},\ldots,v_{n-1}\}\subset(1,\kappa)$} & \multirow{3}[4]{*}{5} & \multicolumn{3}{c|}{$\{v_{2},\ldots,v_{n/5}\}\subset(1,100)$} \\
			\cmidrule{1-4}    \multirow{2}[2]{*}{2} & \multicolumn{3}{c|}{$\{v_{2},\ldots,v_{n/5}\}\subset(1,100)$} &       & \multicolumn{3}{c|}{$\{v_{n/5+1},\ldots,v_{4n/5}\}\subset(100,\frac{\kappa}{2})$} \\
			& \multicolumn{3}{c|}{$\{v_{n/5+1},\ldots,v_{n-1}\}\subset(\frac{\kappa}{2},\kappa)$} &       & \multicolumn{3}{c|}{$\{v_{4n/5+1},\ldots,v_{n-1}\}\subset(\frac{\kappa}{2}, \kappa)$} \\
			\hline
			\multirow{2}[2]{*}{3} & \multicolumn{3}{c|}{$\{v_{2},\ldots,v_{\frac{n}{2}}\}\subset(1,100)$} & \multirow{2}[2]{*}{6} & \multicolumn{3}{c|}{$\{v_{2},\ldots,v_{10}\}\subset(1,100)$} \\
			& \multicolumn{3}{c|}{$\{v_{n/2+1},\ldots,v_{n-1}\}\subset(\frac{\kappa}{2},\kappa)$} &       & \multicolumn{3}{c|}{$\{v_{11},\ldots,v_{n-1}\}\subset(\frac{\kappa}{2},\kappa)$} \\
			\hline
			\multirow{2}[2]{*}{4} & \multicolumn{3}{c|}{$\{v_{2},\ldots,v_{4n/5}\}\subset(1,100)$} & \multirow{2}[2]{*}{7} & \multicolumn{3}{c|}{$\{v_{2},\ldots,v_{n-10}\}\subset(1,100)$} \\
			& \multicolumn{3}{c|}{$\{v_{4n/5+1},\ldots,v_{n-1}\}\subset(\frac{\kappa}{2},\kappa)$} &       & \multicolumn{3}{c|}{$\{v_{n-9},\ldots,v_{n-1}\}\subset(\frac{\kappa}{2},\kappa)$} \\
			\hline
	\end{tabular}}%
	\label{tab:spectrum}%
\end{table}%

We conduct a comparative analysis of several superior algorithms that share similarities with the RBB paradigm. \cite{Dai2019familyspectralgradient} considered the combination of the BB1 and BB2 step sizes and proposed an adaptive truncation scheme based on the cyclic BB strategy (ATC) as follows  
\begin{equation}\label{ATC}
\alpha_{k}^{ATC}=\begin{cases}
\alpha_{k}^{BB1},\quad&\text{if}\quad\text{mod}(k,\widetilde{m})=0,\\
\widetilde{\alpha}_{k},\quad&\text{otherwise},
\end{cases}
\end{equation} 
where the cyclic length $\widetilde{m}$ is a positive integer and 
\begin{equation*}
\widetilde{\alpha}_{k}=\begin{cases}
\alpha_{k}^{BB1},\quad&\text{if}\quad\alpha_{k-1}\le\alpha_{k}^{BB1},\\
\alpha_{k}^{BB2},\quad&\text{if}\quad\alpha_{k-1}\ge\alpha_{k}^{BB2},\\
\alpha_{k-1},\quad&\text{otherwise}.
\end{cases}
\end{equation*} 
\cite{Ferrandi2023harmonicframeworkstepsize} presented a harmonic Rayleigh quotient
\begin{equation}\label{TBB}
\alpha_{k}^{TBB}(\xi_{k})=\frac{\mathbf{y}_{k-1}^{\T}(\mathbf{y}_{k-1}-\xi_{k}\mathbf{s}_{k-1})}{\mathbf{s}_{k-1}^{\T}(\mathbf{y}_{k-1}-\xi_{k}\mathbf{s}_{k-1})}
\end{equation}
with a target $\xi_{k}\in \mathbb{R}$. In a sense, this can also be regard as a regularization, though it is rooted in $\alpha_{k}^{BB2}$ and seeks to approximate $\alpha_{k}^{BB1}$ by adjusting target. We consider a variants of the \text{ABBmin} method, which we indicate with  \text{ABBbon} \cite{Bonettini2009scaledgradientprojection}. 
\text{ABBbon} is defined in the same way as \text{ABBmin} but with an adaptive threshold value $\nu_{k}$ as follows
\begin{equation}\label{ABBbon}
\nu_{k+1}=\begin{cases}
0.9\nu_{k},\quad&\text{if}\quad\cos^2\theta_{k}<\nu_{k},\\
1.1\nu_{k},\quad&\text{otherwise},
\end{cases}
\end{equation} 
with $\nu_{1}=0.5$. The authors in  \cite{Huang2021EquippingBarzilaiBorwein} proposed a BB-like gradient step size (BBQ) with two-dimensional quadratic termination property, and numerical results demonstrated that the algorithm is very efficient. Therefore, we use BBQ as a comparison algorithm and follow the parameter settings in \cite{Huang2021EquippingBarzilaiBorwein}. As suggested in \cite{Frassoldati2008Newadaptivestepsize}, $\tau=0.8$ and $m=9$ was used for the ABBmin method.

We compare the \text{RBB} method with the BB1, BB2 \eqref{BB steps}, \text{ABB} \eqref{ABB}, \text{ATC} \eqref{ATC}, \text{TBB} \eqref{TBB}, ABBmin \eqref{ABBmin}, ABBbon \eqref{ABBbon}, and BBQ \cite{Huang2021EquippingBarzilaiBorwein} methods, where the parameters of the \text{ABB} and \text{ATC} methods are the same as in \cite{Dai2019familyspectralgradient}, and the parameter of the \text{TBB} method is $\xi_{k}=-\cot\theta_{k}$ as suggested in \cite{Ferrandi2023harmonicframeworkstepsize}. In the \text{ERBB} method \eqref{alternate step}, we set $\varrho=5$. 

We present the performance profiles of these methods on the quadratic problems \eqref{pro:quadra} for each tolerance $\varepsilon$ in Figure \ref{figure:different methods quadra}. As observed, when the required convergence precision is $10^{-6}$, ERBB demonstrates significant superiority compared to other algorithms. These numerical results validate the effectiveness of the proposed alternating criterion and adaptive threshold strategy. When the convergence precision requirement is increased to $10^{-10}$, BBQ achieves the best performance due to its quadratic termination property, followed closely by ERBB, both notably outperforming the ABBbon and ABBmin methods. If computational complexity is not considered, RBBA, which does not employ alternating strategies, shows remarkable performance, achieving results comparable to ABBbon. This evidence indicates that the regularization parameters are effective.

\begin{figure}[!ht]
	\centering
	\subfigure[$\varepsilon=10^{-6}$]{
		\includegraphics[width=0.49\textwidth]{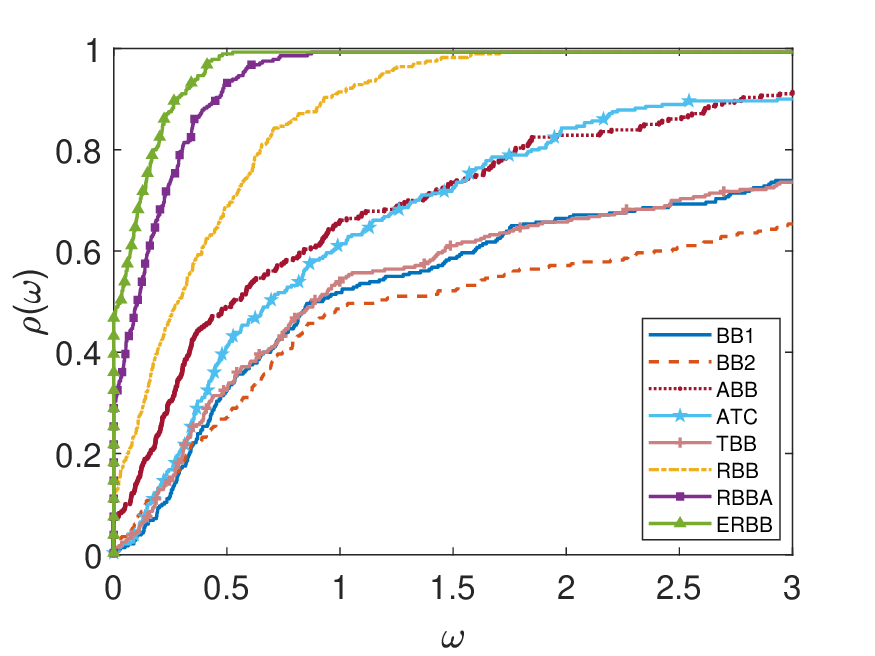}}\hspace{-4pt}
	\subfigure[$\varepsilon=10^{-6}$]{
		\includegraphics[width=0.49\textwidth]{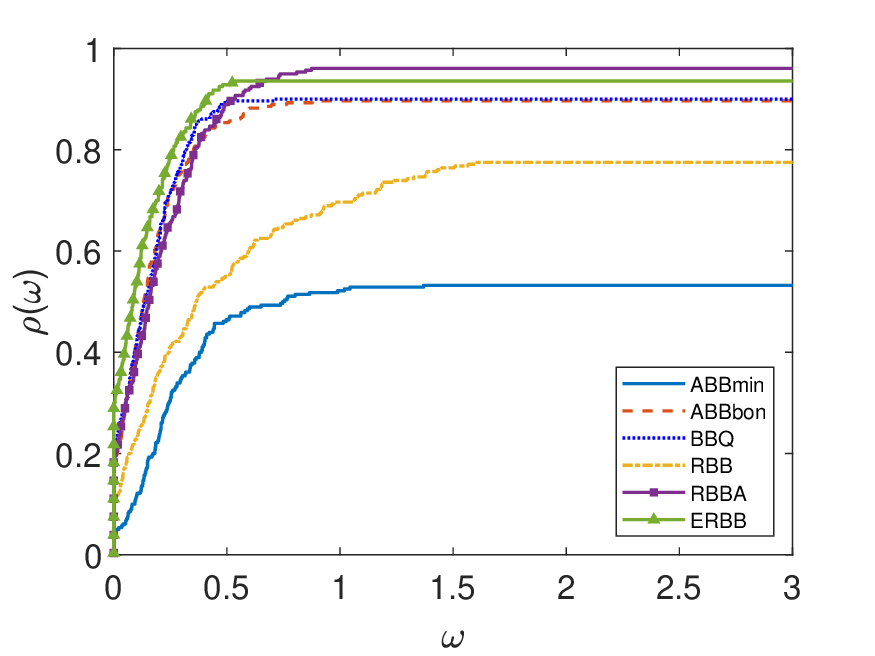}}\\
	\subfigure[$\varepsilon=10^{-10}$]{
		\includegraphics[width=0.49\textwidth]{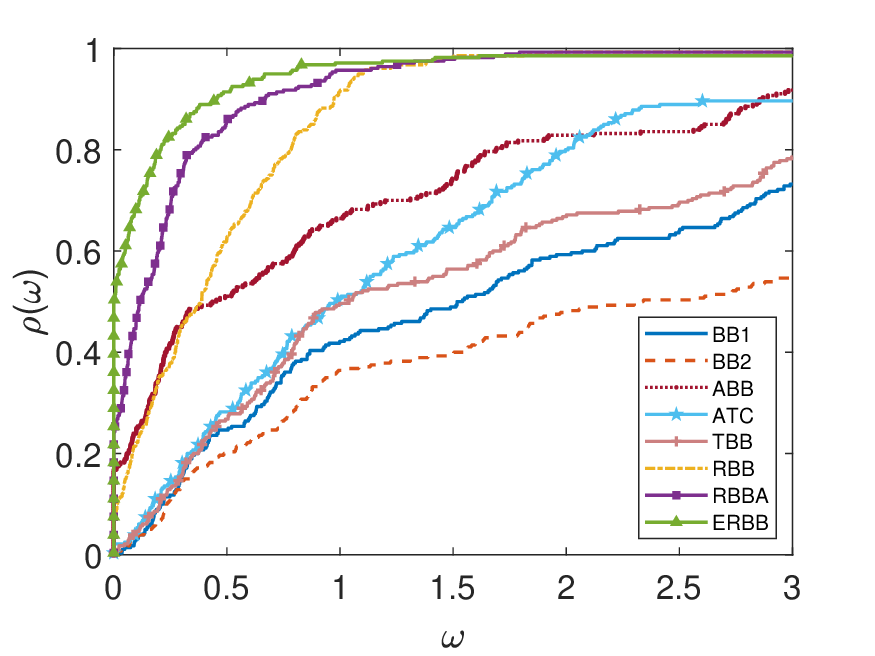}}\hspace{-4pt}
	\subfigure[$\varepsilon=10^{-10}$]{
		\includegraphics[width=0.49\textwidth]{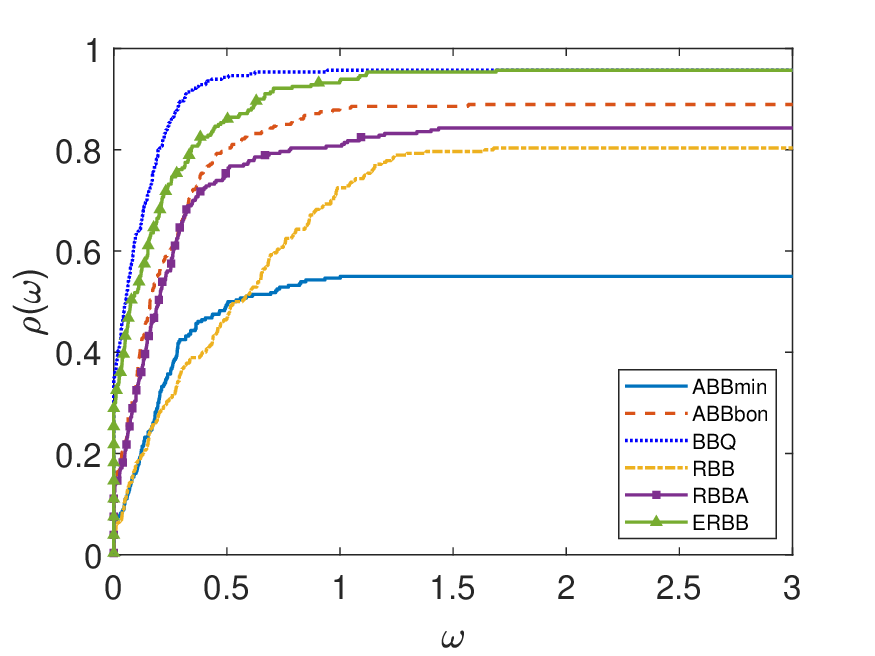}}
	\caption{\textit{Performance profiles of the BB1, BB2, ABB, ATC, TBB, ABBmin, ABBbon, BBQ, RBB, RBBA, ERBB methods on the strictly convex quadratic problems \eqref{pro:quadra} with $n=1000$}, iteration metric}	
	\label{figure:different methods quadra}	
\end{figure}

We next compare these methods on a two-point boundary value problem \cite{Huang2022accelerationBarzilaiBorweinmethod} which can be transferred as a linear system $Ax=b$ by the finite difference method. In particular, the matrix $\mathbf{A}=(a_{i,j})$ is given by
\begin{equation}\label{equ:TBV}
a_{i,j}=\begin{cases}
\frac{2}{h^2},\quad&\text{if}\quad i=j,\\
-\frac{1}{h^2},\quad&\text{if}\quad i=j\pm 1,\\
0,\quad&\text{otherwise},
\end{cases}
\end{equation}
where $h=11/n$. Obviously, $\kappa(\mathbf{A})$ increases as $n$ becomes large. Likewise, we consider the objective function \eqref{pro:quadra}.

The parameter settings of the compared algorithms are consistent with those in the preceding test. We set $n=500, 1000, 1500, 2000, 2500$, and use $x_{1}=\mathbf{1}$ as  starting point. In this practical problem, we set the convergence precisions to $\varepsilon=10^{-4}, 10^{-8}$, respectively. Ten independent runs were performed with these settings. 

\begin{figure}[!ht]
	\centering
	\subfigure[$\varepsilon=10^{-4}$]{
		\includegraphics[width=0.49\textwidth]{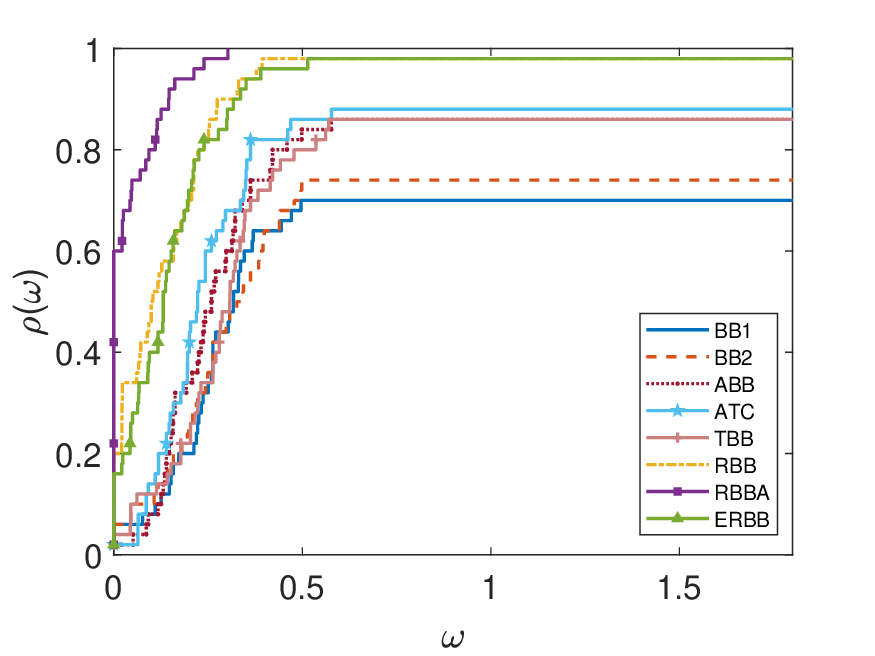}}\hspace{-4pt}
	\subfigure[$\varepsilon=10^{-4}$]{
		\includegraphics[width=0.49\textwidth]{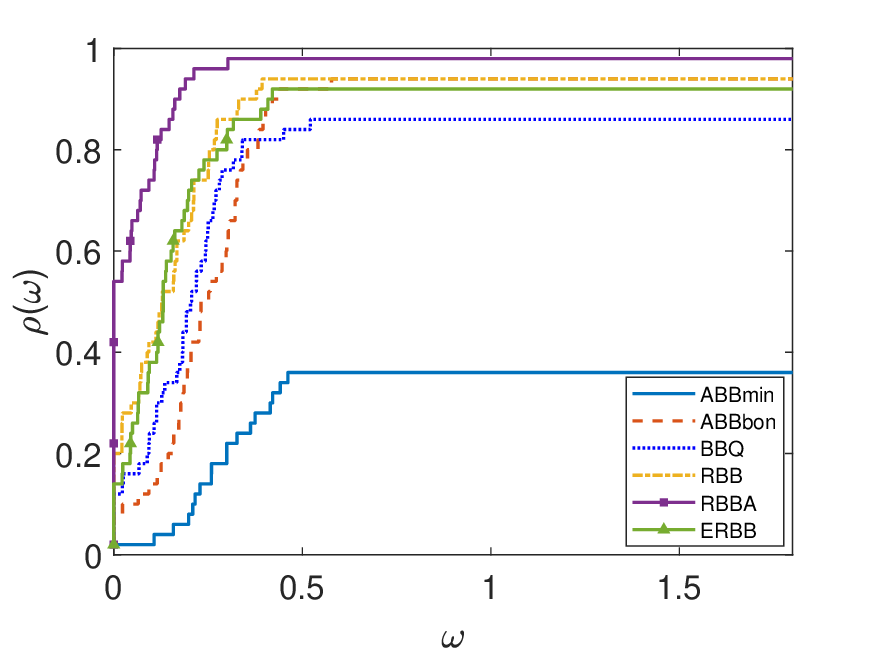}}\\
	\subfigure[$\varepsilon=10^{-8}$]{
		\includegraphics[width=0.49\textwidth]{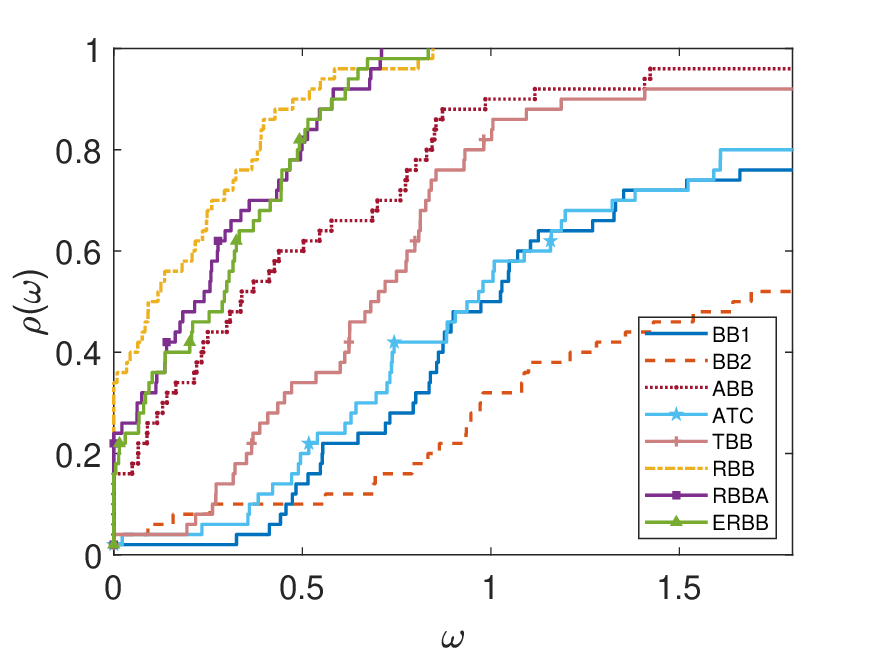}}\hspace{-4pt}
	\subfigure[$\varepsilon=10^{-8}$]{
		\includegraphics[width=0.49\textwidth]{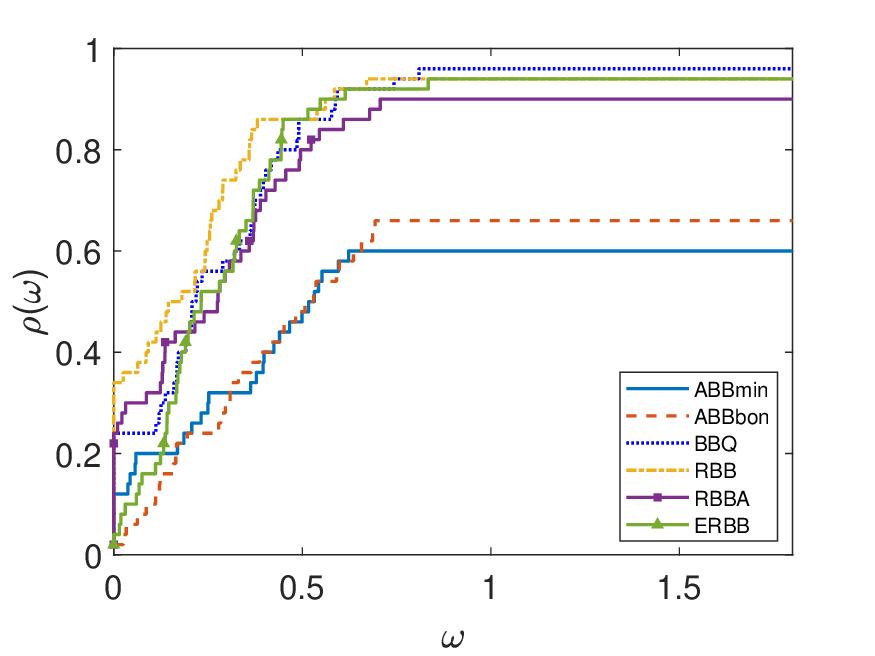}}
	\caption{\textit{Performance profiles of the BB1, BB2, ABB, ATC, TBB, ABBmin, ABBbon, BBQ, RBB, RBBA, ERBB methods on the two-point bound value problems \eqref{equ:TBV}}, iteration metric}	
	\label{figure:two-point bound value}	
\end{figure} 

From the results in Figure \ref{figure:two-point bound value}, it can be observed that at low precision requirement ($\varepsilon=10^{-4}$), RBBA demonstrates the best performance, followed by RBB and ERBB. When the precision increases to $\varepsilon=10^{-8}$, RBB exhibits the best performance.
\subsection{Solving non-quadratic minimization problems}\label{sec:NonQua}
In this subsection, we implement Algorithm \eqref{alg:RBB} to solve general optimization problems. Generally, as described in Section \ref{sec:nonqua}, we need more adjustable parameters than in quadratic problems. We maintain the parameter settings in \cite{Serafino2018TwoPhaseGradient} and set $\alpha_{\min}=10^{-30}$, $\alpha_{\max}=10^{30}$, $\sigma=10^{-4}$, $\delta=\frac{1}{2}$, $M=10$, the maximum number of internal non-monotone line searches to $100$ per iteration. The wide step size bound is to choose BB-like steps as much as possible. Following Raydan \cite{Raydan1997BarzilaiBorweinGradienta}, we choose $\frac{1}{\hat{\alpha}_{k}}=\max\big(\min(\|\mathbf{g}_{k}\|_{2}^{-1}, 10^{5}),1\big)$ as the replacement for negative step size, and use $\frac{1}{\alpha_{1}}=\|\mathbf{x_{1}}\|_{\infty}/\|\mathbf{g}_{1}\|_{\infty}$ if $x_{1}\neq 0$ and otherwise $\frac{1}{\alpha_{1}}=1/\|\mathbf{g}_{1}\|_{\infty}$. We terminate the algorithm when $\|\mathbf{g}_{k}\|_{2}<\varepsilon$ or the number of iterations reaches $20000$ or the number of function evaluations reaches $10^5$. The parameters of each algorithm are consistent with the preceding settings. 

We first consider the classical Rosenbrock function \cite{Rosenbrock1960AutomaticMethodFinding}
\begin{equation}\label{Rosebrock}
	f(\mathbf{x})=c\big(x^{(2)}-(x^{(1)})^2\big)^{2}+(1-x^{(1)})^{2},
\end{equation}
which is often used as a test case for optimization algorithms, where $c$ is a constant that controls the difficulty of the problem, the initial point is $(x_{1}^{(1)}, x^{(2)}_{1})=(-1.2, 1)$. The global minimum is inside a long, narrow, parabolic-shaped flat valley. We compare the performance of these BB-like methods. We set the stop criterion as $\|(x^{(1)}_{k}, x^{(2)}_{k})-(x^{(1)}_{*}, x^{(2)}_{*})\|_{2}<\varepsilon$, where $(x^{(1)}_{*}, x^{(2)}_{*})=(1, 1)$ is the minimizer of $f(x)$. We report in Table \ref{tab:Rosebrock2} the number of iterations for different algorithms with different $\varepsilon$, where the $``-"$ sign indicates that the number of iterations exceeds $9000$. 

\begin{center}
	\setlength{\tabcolsep}{3.5pt}{
		\begin{longtable}[!ht]{cccccccccccc}
			\caption{Performance of the BB1, BB2, ABB, ABBmin, ABBbon, ATC, TBB, BBQ, RBB, ERBB methods on Rosenbrock function}
			\label{tab:Rosebrock2}\\
			\toprule
			\multicolumn{1}{c}{c} & \multicolumn{1}{c}{$\varepsilon$} & \multicolumn{1}{c}{BB1}&
			\multicolumn{1}{c}{BB2}&
			\multicolumn{1}{c}{ABB}&
			\multicolumn{1}{c}{ABBmin}&
			\multicolumn{1}{c}{ABBbon}&
			\multicolumn{1}{c}{ATC}&
			\multicolumn{1}{c}{TBB}&
			\multicolumn{1}{c}{BBQ}&
			\multicolumn{1}{c}{RBB}&
			\multicolumn{1}{c}{ERBB}\\ 
			\midrule
			
			\endfirsthead
			
			\multicolumn{10}{c}%
			{{\bfseries \tablename\ \thetable{} -- continued from previous page}} \\
			\toprule \multicolumn{1}{c}{P} & \multicolumn{1}{c}{$\varepsilon$} & \multicolumn{1}{c}{BB1}&
			\multicolumn{1}{c}{BB2}&
			\multicolumn{1}{c}{ABB}&
			\multicolumn{1}{c}{ABBmin}&
			\multicolumn{1}{c}{ABBbon}&
			\multicolumn{1}{c}{ATC}&
			\multicolumn{1}{c}{TBB}&
			\multicolumn{1}{c}{BBQ}&
			\multicolumn{1}{c}{RBB}&
			\multicolumn{1}{c}{ERBB}\\ 
			\midrule 
			\endhead
			
			\hline \multicolumn{10}{l}{{Continued on next page}} \\ 
			\endfoot
			
			\hline
			\endlastfoot
			\multirow{4}[1]{*}{$10^2$} & $10^{-1}$ & \textbf{36} & 51    & 114   & 56    & 76    & 89    & 899   & 609   & 55    & 74 \\
		    & $10^{-2}$ & \textbf{41} & 57    & 131   & 80    & 82    & 89    & 1726  & 777   & 61    & 103 \\
		    &$10^{-4}$ & \textbf{49} & 63    & 136   & 934   & 260   & 100   & 4647  & 1709  & 67    & 106 \\
		    &$10^{-8}$ & \textbf{53} & 69    & 142   & 934   & 262   & 106   & --    & 4093  & 72    & 184 \\
		\midrule
		\multirow{4}[2]{*}{$10^3$} & $10^{-1}$ & 131   & \textbf{125} & 202   & 163   & 163   & 163   &--     & 7174  & 134   & 176 \\
	    &$10^{-2}$ & 136   & 136   & 202   & 199   & 200   & 169   & --    & --     & \textbf{134} & 224 \\
	    &$10^{-4}$ & 144   & 141   & 214   & 288   & 286   & 175   & --     & --     & \textbf{140} & 247 \\
		&$10^{-8}$ & 148   & 148   & 219   & 288   & 346   & 183   & --    & --     & \textbf{147} & 287 \\
		\midrule
		\multirow{4}[2]{*}[5ex]{$10^4$} & $10^{-1}$ & \textbf{262} & 409   & 479   & 302   & 307   & 462   & --    & --     & 329   & 278 \\
	    &$10^{-2}$ & \textbf{286} & 444   & 499   & 327   & 331   & 500   & --     & --     & 354   & 305 \\
		&$10^{-4}$ & \textbf{291} & 450   & 511   & 411   & 391   & 505   & --     & --     & 359   & 358 \\
		&$10^{-8}$ & \textbf{299} & 480   & 536   & 710   & 754   & 553   & --     & --     & 364   & 448 \\
		\midrule
		\multirow{4}[2]{*}{$10^5$} & $10^{-1}$ & 645   & 634   & 800   & 582   & 582   & 717   & --     & --     & 516   & \textbf{219} \\
		&$10^{-2}$ & 685   & 689   & 850   & 612   & 613   & 784   & --     & --     & 566   & \textbf{250} \\
		&$10^{-4}$ & 696   & 689   & 850   & 714   & 711   & 795   & --     & --     & 571   & \textbf{341} \\
		&$10^{-8}$ & 721   & --     & 866   & 1014  & 975   & 843   & --     & --     & 582   & \textbf{413} \\
	\end{longtable}}
\end{center}

In this problem, when $c=10^2, 10^3$, the number of iterations required by BB1, BB2, and RBB are comparable. When $c=10^4$, BB1 requires the fewest iterations. Whereas when $c=10^5$, ERBB requires the minimal iteration number. Nevertheless, TBB and BBQ need a significantly large number of iterations. 
 
\subsubsection{Test on a collection of unconstrained optimization functions}
For general objective functions, the performance of these methods were tested on a collection of unconstrained minimization problems from \cite{Andrei2008UnconstrainedOptimizationTest} with dimension less than or equal to 5000, which provides a standard starting point $x_{1}$ for each problem, and some of these tests are derived from the CUTEst \cite{More1981TestingUnconstrainedOptimization}. We delete the problem if either it can not be solved in 20000 iterations by any of the algorithms or the function evaluation exceeds $10^5$ and 74 problems are left. The stopping condition $\|\mathbf{g}_{k}\|_{2}\le 10^{-5}$ was adopted for these compared algorithms. In this part, TBB and BBQ methods have similar phenomena as in preceding Rosenbrock function, so we do not show the numerical results for these two.

Performance profiles of these algorithms using the number of iterations and function evaluations metrics are plotted in Figure \ref{figure:CUTE}. From Figure \ref{figure:CUTE}, it can be seen that RBB and ERBB perform significantly better than other compared algorithms, which is attributed to the fact that the regularization parameter includes information about the local morphological changes of the objective function, making them applicable to general objective functions.

\begin{figure}[!ht]
	\centering
	\subfigure{
		\includegraphics[width=0.49\textwidth]{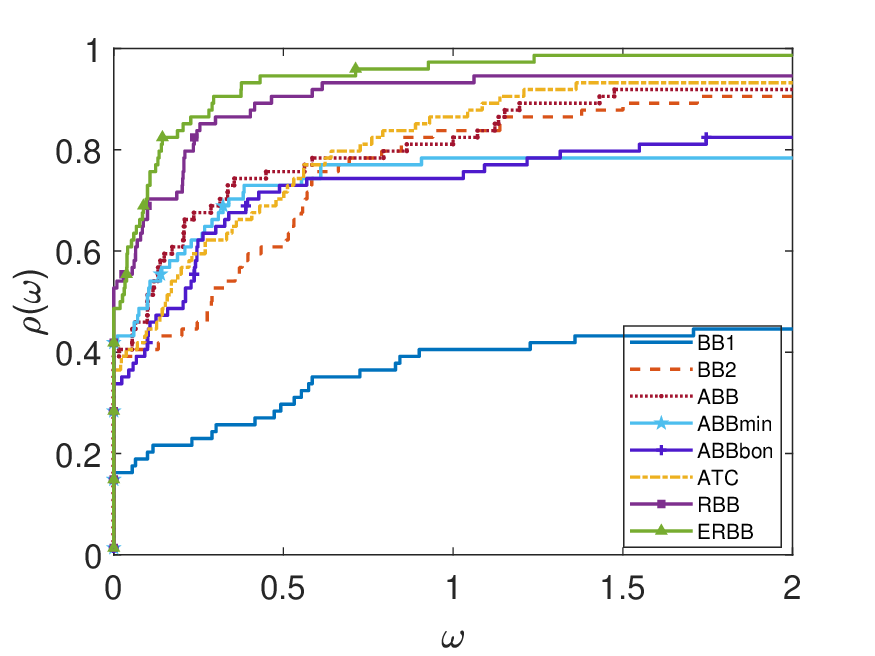}}\hspace{-4pt}
	\subfigure{
		\includegraphics[width=0.49\textwidth]{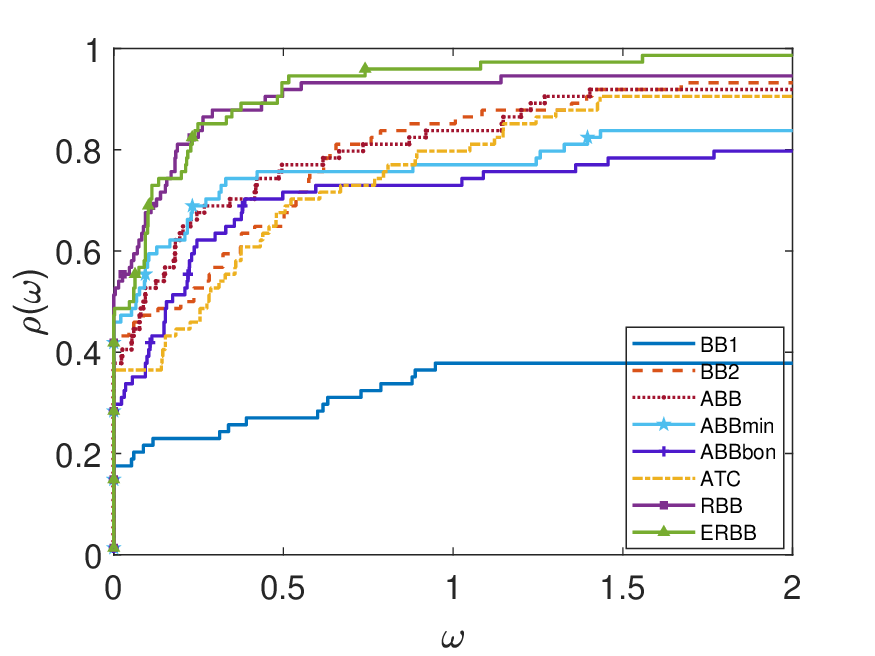}}
	\caption{\textit{Performance profiles of the BB1, BB2, ABB, ABBmin, ABBbon, ATC, RBB, ERBB methods on 74 unconstrained problems from  \cite{Andrei2008UnconstrainedOptimizationTest}}, iteration (left) and function evaluation (right) metrics}	
	\label{figure:CUTE}	
\end{figure}

\subsection{Finding spherical t-designs}
In this part, we consider a numerical computational problem of finding a set of points with ``good" distribution on the unit sphere $\mathbb{S}^{2}:=\{(x,y,z)^{\T}\in\mathbb{R}^{3}|x^2+y^2+z^2=1\}$. A point set $X_N=\{\mathtt{x}_{1},\ldots,\mathtt{x}_{N}\}\subset\mathbb{S}^{2}$ is a spherical $t$-design if it satisfies
\begin{equation*}
\frac{1}{N}\sum_{j=1}^{N}p(\mathtt{x}_{j})=\frac{1}{4\pi}\int_{\mathbb{S}^{2}}p(\mathtt{x})d\omega(\mathtt{x}), \quad \forall \ p\in\mathbb{P}_{t},
\end{equation*}
where $d\omega(\mathtt{x})$ denotes area measure on the unit sphere, $\mathbb{P}_{t}:=\mathbb{P}_{t}(\mathbb{S}^{2})$ is the space of spherical polynomials on $\mathbb{S}^{2}$ with degree at most $t$. For more details on spherical designs, see \cite{Delsarte1991Sphericalcodesdesigns,Congpei2010WellConditionedSpherical,Bannai2009surveysphericaldesigns}. In  \cite{Sloan2009variationalcharacterisationspherical}, the authors present a variational characterization of spherical $t$-design as follows
\begin{equation}\label{variational}
A_{N,t}(X_N):=\frac{4\pi}{N^2}\sum_{j=1}^{N}\sum_{i=1}^{N}\sum_{n=1}^{t}\frac{2n+1}{4\pi}P_n(\langle\mathbf{x}_{j},\mathbf{x}_{i}\rangle),
\end{equation}
where $P_{n}:[-1,1]\rightarrow\mathbb{R}$ is the Legendre polynomial and $\langle\mathbf{x},\mathbf{y}\rangle:=\mathbf{x}^{\T}\mathbf{y}$ is the inner product in $\mathbb{R}^{3}$. It is know that $X_N^{*}$ is a spherical $t$-design if and only if $A_{N,t}(X_N^{*})=0$ \cite{Sloan2009variationalcharacterisationspherical}, and the set of points on the sphere forms an \textit{Oblique} manifold ($\mathcal{OB}(3,N)$) \cite{Absil2008OptimizationAlgorithmsMatrix}. Based on these facts, finding spherical $t$-designs is equivalent to \text{solving} the following matrix optimization problem

\begin{equation}\label{tDesign}
\begin{aligned}
\min \quad   &A_{N,t}(X_{N})\\
\text{s.t.} \quad & X_{N}\subset\mathcal{OB}(3,N),
\end{aligned}
\end{equation}
where $\mathcal{OB}(3,N):=\{X\in\mathbb{R}^{3\times N}: \text{ddiag}(X^{\T}X)=\mathbf{I}_{N}\}$, where $\text{ddiag}(\mathbf{A})$ denotes the matrix $\mathbf{A}$ with all its off-diagonal elements assigned to zero, $\mathbf{I}_{N}$ is the $N\times N$ identity matrix.

Given that the sphere has good geometric properties, problem \eqref{tDesign} is actually an unconstrained optimization problem on a matrix manifold. Therefore, for finding spherical $t$-designs, \cite{An2020Numericalconstructionspherical} numerically construct spherical $t$-designs by using BB method. Based on the code in \cite{An2020Numericalconstructionspherical}, we perform the numerical experiments in this subsection. The termination condition of the algorithms is as follows
\begin{equation*}
\|g_{k}\|_{2}<\varepsilon_{1}\|g_{1}\|_{2} \quad \text{or}\quad \|f(x_{k})-f(x_{k-1})\|_{2}<\varepsilon_{2},
\end{equation*}
where $\varepsilon_{1}=10^{-8}$, $\varepsilon_{2}=10^{-16}$. We set the initial step size to $1$. The maximum number of iterations and function evaluations are $10000$ and $20000$, respectively. The parameter settings of these comparison methods are the same as those in the preceding experiments. The initial points in this experiment are consistent with those in \cite{An2020Numericalconstructionspherical}. From the work of Chen et al. \cite{Chen2010Computationalexistenceproofs}, we know that spherical $t$-designs with $N=(t+1)^2$ points exist for all degree $t$ up to $100$ on $\mathbb{S}^{2}$. This encourage us to find higher degree $t$ for spherical $t$-designs. Let $t\ge2$ and $N\ge(t+2)^2$. Assume $X_{N}\subset\mathbb{S}^{2}$ is a stationary point set of $A_{N,t}$ and the minimal singular value of basis matrix $Y_{t+1}(X_{N})$ is positive. Then $X_{N}$ is a spherical $t$-design \cite[Thm.2.4]{An2020Numericalconstructionspherical}. 

Table \ref{tab:spherical} presents the numerical results of these BB-like methods, where $nf$ refers to the number of function evaluations when the termination condition is met, $m(\sigma)$ represents the minimum singular value of the basis system formed $Y_{t+1}(X_{N}^{*})$ by $X_{N}^{*}$, $G^{*}$ denotes the obtained $\|\nabla A_{N,t}(X_{N}^{*})\|_{2}$, and $F^{*}$ refers to the obtained $A_{N,t}(X_{N}^{*})$.

\begin{table}[h]
	\centering
	\caption{Computing spherical $t$-designs with $N=(t+1)^2$ by BB-like methods}
	\setlength{\tabcolsep}{1.5pt}{
		\begin{tabular}{|c|c|cccccccccc|}
			\hline
			& t     & BB1   & BB2   & ABB   & ABBmin & ABBbon & ATC   & TBB   & BBQ   & RBB   & ERBB \\
			\hline
		\multirow{5}[2]{*}{nf} & 10    & 86    & 63    & 56    & 58    & 56    & 70    & 76    & 59    & 57    & 61 \\
		& 50    & 238   & 117   & 160   & 118   & 112   & 136   & 139   & 111   & 107   & 125 \\
		& 70    & 196   & 180   & 149   & 161   & 156   & 205   & 166   & 123   & 133   & 142 \\
		& 90    & 2490  & 1751  & 1123  & 1506  & 1353  & 1417  & 2022  & 1133  & 1131  & 1335 \\
		& 130   & 567   & 212   & 206   & 241   & 211   & 309   & 307   & 183   & 225   & 196 \\
		\hline
		\multirow{5}[2]{*}{m($\sigma$)} & 10    & 1.3148  & 1.3149  & 1.3148  & 1.3148  & 1.3148  & 1.3148  & 1.3148  & 1.3148  & 1.3148  & 1.3148  \\
		& 50    & 2.3558  & 2.3558  & 2.3558  & 2.3557  & 2.3557  & 2.3558  & 2.3558  & 2.3558  & 2.3557  & 2.3558  \\
		& 70    & 2.1394  & 2.1395  & 2.1394  & 2.1393  & 2.1393  & 2.1394  & 2.1394  & 2.1393  & 2.1393  & 2.1393  \\
		& 90    & 0.0037  & 0.0013  & 0.0192  & 0.0048  & 0.0059  & 0.0035  & 0.0180  & 0.0117  & 0.0005  & 0.0053  \\
		& 130   & 2.0870  & 2.0872  & 2.0871  & 2.0870  & 2.0870  & 2.0871  & 2.0871  & 2.0870  & 2.0871  & 2.0871  \\
		\hline
		\multirow{5}[2]{*}{$G^{*}$} & 10    & 7.6E-08 & 6.9E-08 & 6.2E-08 & 4.9E-08 & 6.6E-08 & 2.3E-08 & 3.8E-10 & 1.0E-08 & 8.2E-08 & 9.0E-08 \\
		& 50    & 3.9E-07 & 2.2E-07 & 3.9E-07 & 3.6E-07 & 3.9E-07 & 3.4E-07 & 3.7E-07 & 3.1E-07 & 2.9E-07 & 3.0E-07 \\
		& 70    & 4.8E-07 & 5.4E-07 & 4.7E-07 & 5.3E-07 & 3.9E-07 & 2.7E-07 & 5.1E-07 & 4.4E-07 & 4.8E-07 & 5.3E-07 \\
		& 90    & 5.2E-07 & 5.5E-07 & 7.1E-07 & 7.1E-07 & 7.7E-07 & 5.7E-07 & 5.4E-07 & 6.6E-07 & 7.4E-07 & 6.0E-07 \\
		& 130   & 9.8E-07 & 7.6E-07 & 8.6E-07 & 9.5E-07 & 8.8E-07 & 6.5E-07 & 9.6E-07 & 9.9E-07 & 9.6E-07 & 8.0E-07 \\
		\hline
		\multirow{5}[2]{*}{$F^{*}$} & 10    & 8.3E-14 & 4.6E-15 & 3.4E-14 & 4.2E-15 & 3.3E-14 & 8.1E-15 & 2.3E-15 & 2.3E-15 & 5.9E-14 & 8.5E-14 \\
		& 50    & 1.3E-11 & 8.5E-12 & 5.6E-12 & 1.3E-11 & 1.6E-11 & 1.1E-11 & 1.4E-11 & 9.6E-12 & 9.0E-12 & 7.8E-12 \\
		& 70    & 1.7E-11 & 3.5E-11 & 1.3E-11 & 3.3E-11 & 1.1E-11 & 7.3E-12 & 2.9E-11 & 1.2E-11 & 1.9E-11 & 1.9E-11 \\
		& 90    & 1.7E-11 & 3.0E-10 & 8.6E-11 & 4.1E-10 & 9.3E-11 & 2.1E-11 & 2.7E-10 & 2.3E-10 & 4.0E-10 & 3.6E-11 \\
		& 130   & 3.4E-10 & 1.9E-10 & 9.0E-11 & 2.4E-10 & 1.7E-10 & 5.0E-11 & 1.9E-10 & 1.8E-10 & 8.9E-11 & 8.8E-11 \\
	\hline
\end{tabular}}%
\label{tab:spherical}%
\end{table}%

From the data in Table \ref{tab:spherical}, we can see that in terms of the number of function value evaluations, the performance of RBB is comparable to that of BBQ, and that of ERBB is comparable to that of ABB, and their performance is significantly better than the rest of the algorithms. In the case of $t=130$, it can be seen from the numerical results that $\min(\sigma)>0$ holds, so the obtained $X_{N}^{*}$ by these algorithms is non-singular, which indicates that $X_{N}^{*}$ is a spherical $t$-design.

\section{Conclusion}\label{conclusion}
In this paper, we first propose a regularized least squares model, which provides a modification of the original BB1 method, and then consider two types of $\Phi(A)$ for the quadratic optimization problem, which correspond to two types of spectral gradient step sizes. We provide a mathematical explanation of the ABB method. Based on this and the local mean curvature of the objective function, we propose a three-step regularization parameter scheme that enables the RBB step size to quickly approximate the BB1 or BB2 step size. Based on the analysis of the ABB method, we find a new alternate step size criterion and incorporate this criterion into ABBmin, thereby obtaining an enhanced RBB method. How to select more efficient regularization parameters and analyze ABB when $n\ge 3$ require further exploration.

\section{Declaration}
\subsection{Ethical Approval}
Not Applicable.
\subsection{Availability of supporting data}
The data are all included in the paper.
\subsection{Competing interests}
The authors declare that there is no conflict of interest.
\subsection{Funding}
The work was supported by the National Natural Science Foundation of China (Project No. 12371099).
\subsection{Authors' contribution}
In this work, Xin Xu and Congpei An contributed equally, with Xin Xu acting as the corresponding author.

\bibliography{RBB}

\end{document}